\documentclass[11pt,reqno]{amsart}
\usepackage{fourier}
\usepackage{fullpage}
\usepackage{mathrsfs,amssymb,graphicx,verbatim,amsmath,amsfonts}
\usepackage{paralist}
\usepackage[breaklinks,pdfstartview=FitH]{hyperref}
\usepackage{upgreek}
\usepackage[mathscr]{euscript}

\usepackage{xcolor}

\makeatletter
\def\resetMathstrut@{%
  \setbox\z@\hbox{%
    \mathchardef\@tempa\mathcode`\(\relax
    \def\@tempb##1"##2##3{\the\textfont"##3\char"}%
    \expandafter\@tempb\meaning\@tempa \relax
  }%
  \ht\Mathstrutbox@1.2\ht\z@ \dp\Mathstrutbox@1.2\dp\z@
}
\makeatother
\renewcommand{\le}{\leqslant}
\renewcommand{\ge}{\geqslant}
\renewcommand{\leq}{\leqslant}
\renewcommand{\geq}{\geqslant}
\renewcommand{\setminus}{\smallsetminus}
\newcommand{\ug}{\upgamma}
\newcommand{\ur}{\uprho}

\newcommand{\ud}[0]{\,\mathrm{d}}

\usepackage{dutchcal}

\renewcommand{\det}{\mathrm{det}}

\newcommand{\B}{\mathsf{B}}
\newcommand{\n}{\{1,\ldots,n\}}
\newcommand{\m}{\{1,\ldots,m\}}

\newcommand{\X}{\mathbf X}

\newcommand{\sfC}{\mathsf{C}}
\newcommand{\f}{\varphi}

\newcommand{\T}{\mathsf{T}}
\newcommand{\sY}{\mathsf{Y}}
\renewcommand{\d}{\delta}

\newcommand{\e}{\varepsilon}
\newcommand{\R}{\mathbb R}

\newcommand{\1}{\mathbf 1}

\newtheorem{theorem}{Theorem}
\newtheorem{lemma}[theorem]{Lemma}

\newtheorem{remark}[theorem]{Remark}

\newtheorem{conjecture}[theorem]{Conjecture}

\newcommand{\ut}{\uptau}
\newcommand{\up}{\upphi}
\newcommand{\ub}{\upbeta}

\newcommand{\Ch}{\mathrm{Ch}}
\newcommand{\Per}{\mathrm{Per}}

\renewcommand{\S}{\mathsf{S}}
\renewcommand{\subset}{\subseteq}

\newcommand{\I}{\mathsf{I}}

\newcommand{\N}{\mathbb N}
\newcommand{\Z}{\mathbb Z}

\newcommand{\eqdef}{\stackrel{\mathrm{def}}{=}}

\newcommand{\A}{\mathsf{A}}
\newcommand{\sC}{\mathsf{C}}
\newcommand{\V}{\mathsf{V}}

\renewcommand{\emptyset}{\varnothing}

\newcommand{\conv}{\mathrm{conv}}

\newcommand{\iq}{\mathrm{iq}}
\newcommand{\Id}{\mathsf{I}}

\newcommand{\M}{\mathsf{M}}
\newcommand{\GL}{\mathsf{GL}}
\newcommand{\SL}{\mathsf{SL}}
\renewcommand{\O}{\mathsf{O}}

\newcommand{\vol}{\mathrm{vol}}

\begin{document}

\title{Approximate isoperimetry for convex polytopes}\thanks{K.~J.~B. was supported by NKKP grant 150613.   A.N.~ was supported by NSF grant DMS-2054875, BSF grant  201822, and a Simons Investigator award. }

\author{Keith Ball}
\address{Mathematics Institute\\ University of Warwick\\ Coventry CV4 7AL, UK}
\email{k.m.ball@warwick.ac.uk}

\author{K\'aroly J. B\"or\"oczky}
\address{Alfr\'ed R\'enyi Institute of Mathematics,\\ Realtanoda u.~13-15, H-1053 Budapest, Hungary}
\email{boroczky.karoly.j@renyi.hu}

\author{Assaf Naor}
\address{Mathematics Department\\ Princeton University\\ Fine Hall, Washington Road, Princeton, NJ 08544-1000, USA}
\email{naor@math.princeton.edu}

\maketitle

\vspace{-0.3in}

\begin{abstract} For all $n,\up\in \N$ with $\up\ge n+1$, the smallest possible isoperimetric quotient of an $n$-dimensional convex polytope that has $\up$ facets is shown to be bounded from above and from below by positive universal constant multiples of $\max\big\{n/\sqrt{1+\log (\up/n)},\sqrt{n}\big\}$. For all $n\in \N$ and $2n\le \ub\in 2\N$, it is shown that every $n$-dimensional origin-symmetric convex polytope that has $\ub$ vertices admits an affine image whose isoperimetric quotient is at most a  universal constant multiple of $\min\big\{\sqrt{\log(\ub/n)},n\big\}$, which is  sharp. The weak isomorphic reverse isoperimetry conjecture is proved for $n$-dimensional convex polytopes that have $O(n)$ facets by demonstrating that any such polytope $K$ has an image $K'$ under a volume preserving matrix and a convex body $L\subset K'$ such that the isoperimetric quotient of $L$ is at most a universal constant multiple of $\sqrt{n}$, and also $\sqrt[n]{\vol_n(L)/\vol_n(K)}$ is at least a positive universal constant. 
\end{abstract}

\section{Introduction}

Fix $n\in \N$. Whenever  a convex polytope in $\R^n$ will be mentioned below, it will be assumed tacitly that it is also a convex body, namely, it is compact and has nonempty interior. The isoperimetric quotient\footnote{The literature often uses the terminology and notation that we adopt herein (e.g.~\cite{Sch89}), but it is also common for the $n$'th power of the right hand side of~\eqref{eq:def iq} to be called  the isoperimetric quotient of $K$ (e.g.~\cite[page~269]{Had57} or~\cite[page~203]{Gru07}).} of a convex body $K\subset \R^n$ is defined to be the following scale-invariant quantity: 
\begin{equation}\label{eq:def iq}
\iq(K)=\frac{\vol_{n-1}(\partial K)}{\vol_n(K)^\frac{n-1}{n}},
\end{equation}
where $\vol_n$ denotes the Lebesgue measure on $\R^n$ and $\vol_{n-1}$ denotes the surface area measure on $\R^n$. 

By the  classical isoperimetric theorem, the isoperimetric quotient of any convex body $K\subset \R^n$ is at least the isoperimetric quotient of the Euclidean ball, i.e.,  the following lower bound on $\iq(K)$ holds: 
\begin{equation}\label{eq:quote isoperimetric theorem}
\iq(K)\geq \iq(B_{\ell^n_{\!\!2}})=\frac{n\sqrt{\pi}}{\Gamma\left(\frac{n}{2}+1\right)^{\frac{1}{n}}}\asymp\sqrt{n}.
\end{equation}
In~\eqref{eq:quote isoperimetric theorem}, as well as throughout what follows,  $B_\X=\{x\in \R^n:\ \|x\|_\X\le 1\}$ denotes the unit ball of a normed space $\X=(\R^n,\|\cdot\|_\X)$. The space $\ell_{\!\!2}^n$ is $\R^n$ equipped with the standard scalar product  $\langle x,y\rangle =x_1y_1+\ldots x_ny_n$ for  $x=(x_1,\ldots,x_n), y=(y_1,\ldots, y_n)\in \R^n$.  We will also write $B^n=B_{\!\ell_{\!\!2}^n}$ and $S^{n-1}=\partial B^n$. In addition to  the usual $O(\cdot),o(\cdot),\Omega(\cdot), \Theta(\cdot)$ asymptotic notation, we use in~\eqref{eq:quote isoperimetric theorem}, as well as  throughout  the ensuing discussion, the following common conventions for asymptotic notation: Given $a,b\ge 0$, by writing
$a\lesssim b$ or $b\gtrsim a$ we mean that $a\le C b$ for some
universal constant $C>0$, and $a\asymp b$
stands for $(a\lesssim b) \wedge  (b\lesssim a)$.

The following theorem shows how we may improve the asymptotic estimate $\iq(K)\gtrsim \sqrt{n}$ of~\eqref{eq:quote isoperimetric theorem}  if one imposes the further restriction that, rather than being an arbitrary convex body, $K$ is a convex polytope with a fixed number of facets; see Section~\ref{sec:history} below for the history of such investigations.

\begin{theorem}\label{thm:isoperim for poly} Fix  $n,\up\in \N$ with $\up\ge n+1$. Every convex polytope $K\subset \R^n$   that has  $\up$ facets satisfies:
\begin{equation}\label{eq:iq lower in thm}
\iq(K)\gtrsim \max\Bigg\{\frac{n}{\sqrt{1+\log \frac{\upphi}{n}}}, \sqrt{n}\Bigg\}. 
\end{equation}
Furthermore, there exists a convex polytope in $\R^n$ that has  $\up$ facets whose isoperimetric quotient is at most a universal constant multiple of the right hand side of~\eqref{eq:iq lower in thm}. 
\end{theorem}

By considering Cartesian products of cross-polytopes, it is not hard to convince oneself that for many values of $n$ and $\up$ the lower bound on $\iq(K)$ in~\eqref{eq:iq lower in thm} cannot be improved for some $K$. For every $m\in \N$, the cross-polytope $B_{\ell_{\!\!1}^m}=\{x=(x_1,\ldots,x_m)\in \R^m: |x_1|+\ldots+|x_m|\le 1\}$ satisfies $\iq(B_{\ell_{\!\!1}^m})\asymp \sqrt{m}$. Given $k\in \N$, the $k$-fold Cartesian product  $(B_{\ell_{\!\!1}^m})^k$ is an $n$-dimensional convex polytope for $n=km$ that has $\up=k2^m$ facets and whose isoperimetric quotient equals $k\iq(B_{\ell_{\!\!1}^m})\asymp k\sqrt{m}=n/\sqrt{m}\asymp n/\sqrt{1+\log(\up/n)}$. 

So, the main content of Theorem~\ref{thm:isoperim for poly} is to demonstrate that the above explicit and elementary computation for Cartesian products of cross-polytopes produces the worst-possible behavior (up to universal constant factors) among all $n$-dimensional convex polytope that have $\up$ facets, namely,~\eqref{eq:iq lower in thm}   holds for any such polytope whatsoever. Our proof of this statement consists of  a quick concatenation of  (substantial) results that are available in the literature; its details appear in Section~\ref{sec:proof1} below.  

\begin{remark} {\em One could also wonder about improving the classical estimate $\iq(K)\gtrsim \sqrt{n}$ in~\eqref{eq:quote isoperimetric theorem} when $K$ is restricted to be a convex polytope that has $\ub$ vertices for some integer $\ub\ge n+1$. If $\ub\ge 2n$, then this is impossible since $B_{\ell_{\!\!1}^n}$ has $2n$ vertices and  $\iq(B_{\ell_{\!\!1}^n})\asymp \sqrt{n}$, so by considering the convex hull of $B_{\ell_{\!\!1}^n}$  with additional $\ub-2n$ points from $\R^n\setminus B_{\ell_{\!\!1}^n}$ that are in general position and arbitrarily close to $\partial B_{\ell_{\!\!1}^n}$, one sees that there is a convex  polytope in $\R^n$ that has $\ub$ vertices and whose isoperimetric quotient is of order $\sqrt{n}$. If $\up=n+1$, then a stronger isoperimetric lower bound holds since a convex polytope $K\subset \R^n$ that has $n+1$ vertices is a simplex, whence $\iq(K)\ge \iq(\triangle_n) \asymp n$ by~\cite{Had57} (this also follows from~\cite{Pet61}),  where $\triangle_n\subset \R^n$ is the regular simplex. It remains open to determine how to interpolate between the aforementioned asymptotic lower bounds on $\iq(K)$ when $K$ is a convex polytope that has $\ub$ vertices and  $n+1<\ub<2n$.

More generally, one could ask how the estimate $\iq(K)\gtrsim \sqrt{n}$ in~\eqref{eq:quote isoperimetric theorem} improves when $K$ is  a convex polytope that has a fixed number of faces of a given intermediate dimension, or even while fixing its numbers of faces whose dimensions belong to given subset of $\{0,\ldots,n-1\}$; to the best of our knowledge, this question has not been broached in the literature. } 
\end{remark}

\smallskip

The reverse isoperimetric theorem~\cite{Bal91c} states that  every convex body $K\subset \R^n$ has an affine image whose isoperimetric quotient  is at most the isoperimetric quotient of the regular simplex $\triangle_n\subset \R^n$. Thus:
\begin{equation}\label{eq:state reverse iso}
\partial_K\eqdef \min_{\A\in \SL_n(\R)}\iq(\A K)\le \iq(\triangle_n)\asymp n,
\end{equation}
where $\SL_n(\R)$ denotes (as usual) the group of linear transformations of $\R^n$ whose determinant is $1$, and we adopt in~\eqref{eq:state reverse iso} the notation that was introduced in~\cite{GP99} for the affinely invariant quantity $\partial_K$.

 For every $n,\ub\in \N$ with $\ub\ge n+1$ there is a convex polytope $K\subset \R^n$ that has $\ub$ vertices which satisfies $\iq(\A K)\gtrsim n$ for every $\A\in \SL_n(\R)$. Similarly,   for every $n,\up\in \N$ with $\up\ge n+1$ there is a convex polytope $K\subset \R^n$ that has $\up$ facets which satisfies $\iq(\A K)\gtrsim n$ for every $\A\in \SL_n(\R)$. Both of these statements follow from a compactness argument by considering suitable perturbations of the regular simplex $\triangle_n$, which satisfies $\partial_{\triangle_n}\asymp n$; the details appear in Section~\ref{sec:proof2} below.  
 
It is thus impossible to obtain an improved reverse isoperimetric theorem in the spirit of Theorem~\ref{thm:isoperim for poly} for convex polytopes which either have  a given number of vertices or a given number of facets. Nevertheless, a  markedly different reverse isoperimetric phenomenon holds if one considers origin-symmetric convex polytopes $K\subset \R^n$ (or translates thereof), namely, those for which $-K=K$.   

By~\cite{Bal91c},  every origin-symmetric convex body $K\subset \R^n$ satisfies $\partial_K\le \partial_{[0,1]^n}=2n$, improving~\eqref{eq:state reverse iso} optimally by merely a universal constant factor, as $\partial_{\triangle_n}=(e-o(1))n$.  This might lead one to expect that the difference between the general case and the origin-symmetric case remains of this lower-order nature even when one restricts to polytopes that have either  a given number of vertices or a given number of facets. This indeed holds for polytopes with restricted number of facets, i.e., for every $\up\in 2\N$ with $\up\ge 2n$ there is an origin-symmetric convex polytope $K\subset \R^n$ that has $\up$ facets yet $\iq(\A K)\gtrsim n$ for every $\A\in \SL_n(\R)$, as seen by considering  perturbations of the hypercube $[-1,1]^n$ (details are provided in Section~\ref{sec:proof2}). However, for origin-symmetric convex polytopes with a given number of vertices, we have the following sharp asymptotic improvement over the upper bound $\partial_K\lesssim n$ in~\eqref{eq:state reverse iso}:

\begin{theorem}\label{thm:vertex reverse is} Suppose that $n\in \N$ and that $\ub\in 2\N$ satisfies $\ub\ge 2n$. Then, for every origin-symmetric  convex polytope $K\subset \R^n$  that has $\ub$ vertices there exists  $\A\in \SL_n(\R)$ such that: 
\begin{equation}\label{eq:reverse isoperim in thm}
\iq(\A K)\lesssim \min\bigg\{ {\textstyle \sqrt{n\log\frac{\upbeta}{n}}},n\bigg\}.  
\end{equation}
Furthermore,  there exists an origin-symmetric convex polytope $K\subset \R^n$ that has  $\ub$ vertices such that $\iq(\A K)$ is at least a universal constant multiple of the right hand side of~\eqref{eq:reverse isoperim in thm} for every $\A\in \SL_n(\R)$. 
\end{theorem}

As for Theorem~\ref{thm:isoperim for poly}, our proof of~\eqref{eq:reverse isoperim in thm} is a quick concatenation of  (substantial) results that are available in the literature; its details appear in Section~\ref{sec:proof2} below.

\smallskip

A conjectural~\cite{Nao24}  reverse isoperimetric phenomenon asserts that if in addition to judiciously choosing an affine image one is permitted to pass to a certain $O(1)$-perturbation of a given convex body $K\subset \R^n$, then it is always possible to arrive at a convex body whose isoperimetric quotient is $O(\sqrt{n})$, i.e., after both a prudent choice of basis and a bounded correction, {\em every} $n$-dimensional convex body behaves (in terms of isoperimetry) up to positive universal constant factors like   the Euclidean ball. A precise statement is:

\begin{conjecture}[weak isomorphic reverse isoperimetry]\label{conj:weak reverse iso} For every origin-symmetric convex body $K\subset \R^n$ there exist $\A\in \SL_n(\R)$ and an origin-symmetric convex body $L\subset \A K$ that satisfies: 
\begin{equation*}\label{eq:conclusion of reverse iso}
\vol_{n}(L)^{\frac{1}{n}}\gtrsim \vol_{n}(K)^{\frac{1}{n}}\qquad\mathrm{and}\qquad  \iq(L)\lesssim \sqrt{n}. 
\end{equation*}
\end{conjecture}

The term ``weak'' is used in~\cite{Nao24} to name Conjecture~\ref{conj:weak reverse iso}  because~\cite{Nao24} also formulates a stronger conjectural phenomenon; we do not need to recall that stronger conjecture herein since  the present article does not address it, and furthermore the above weaker version suffices for certain applications in nonlinear functional analysis and theoretical computer science; see~\cite{Nao24} for details.

The partial results towards Conjecture~\ref{conj:weak reverse iso} that are currently known can be found in~\cite{Nao24,GN25}. Here, we will prove the following statement:

\begin{theorem}\label{thm:weak reverse iso} Fix $n,\up\in \N$ with $\up\ge n+1$. For every convex polytope  that has $\up$  facets there are a vector $z\in \R^n$, a matrix $\A\in \SL_n(\R)$, and an origin-symmetric convex body $L\subset z+ \A K$ that satisfies: 
\begin{equation}\label{eq:desired conditions for isomorphic in thm}
\vol_n(L)^{\frac{1}{n}}\gtrsim \frac{n}{\up}\vol_n(K)^{\frac{1}{n}}\qquad \mathrm{and}\qquad \iq(L)\lesssim \frac{\up}{\sqrt{n}}. 
\end{equation}
\end{theorem}

Theorem~\ref{thm:weak reverse iso}   shows that Conjecture~\ref{conj:weak reverse iso} holds when $K\subset \R^n$ is a convex polytope that has $\up=O(n)$ facets. Indeed, if $K\subset \R^n$ is an origin-symmetric convex body, then the body $L$ that  Theorem~\ref{thm:weak reverse iso}   provides is, in fact, contained in $\A K$ (not only in its translate $z+\A K$) by the following straightforward argument:
$$
L=\frac12 (L+L)=\frac12(L-L)\subset \frac12 \big ((z+\A K)-(z+\A K)\big )=\frac12 (\A K-\A K)=\frac12 (\A K+\A K)=\A K,  
$$
where we used the convexity of $L$ and $\A K$ in, respectively, the first and last steps,  and we used the fact that $L$ and $\A K$ are origin-symmetric in, respectively, the second and penultimate steps. While this is modest  evidence for Conjecture~\ref{conj:weak reverse iso}, it was previously unknown, and its justification is not entirely trivial. The proof of Theorem~\ref{thm:weak reverse iso}  appears in Section~\ref{sec:proof3} below. It proceeds through a reformulation of Conjecture~\ref{conj:weak reverse iso} from~\cite[Section~1.6.1]{Nao24}  in terms of the spectrum of the  Laplacian on $K$ with Dirichlet boundary conditions, and it yields a spectral bound that improves asymptotically over the best-known bound~\cite{Nao24} if $\up =o(n\log n)$.

\subsection{Historical comments}\label{sec:history} It is quite curious that the  statement of Theorem~\ref{thm:isoperim for poly} has not been previously obtained in the literature. One possible explanation is that while such questions have been studied for a very long time, this was done with the aim to determine the exact isoperimetric minimizers (constrained to have a fixed number of facets), which turns out to be  an extremely difficult (perhaps even hopeless) goal. Aiming to understand the phenomenon up to universal constant factors opens the door to a powerful toolkit that has been developed over the past decades in the local theory of Banach spaces; modulo such known results, our proof of Theorem~\ref{thm:isoperim for poly} is short.  In one of our motivations for examining the question that Theorem~\ref{thm:isoperim for poly} answers (see below), universal constant factors do not matter. 

The extremal property of balls with respect to the isoperimetric problem was known to the ancient Greeks; for example, Zenodorus suggested an argument first proving that regular polygons are optimal in the plane, and even claimed that Euclidean balls are optimal in three dimensions; see~\cite{Bla05} for the history.
In higher dimensions,  the isoperimetric inequality for convex bodies  was proved  by the works of Steiner, Schwarz, Weierstrass and Minkowski (the history is covered in~\cite{Gru07}. Briefly,  Steiner famously provided a symmetrization method  showing that given the volume, only Euclidean balls can be the minimizers of the surface area. However, Steiner did not prove the existence of a minimizer, which was subsequently verified by Weierstrass and Schwarz). Concerning convex polytopes, Zenodorus already suggested that among planar convex polygons that have a  given number of sides, the regular ones have the minimal perimeter, but this was only proved rigorously by Weierstrass; see~\cite{Wei27}. In higher dimensions,
Steiner using his symmetrization method also proved that among simplices of given volume, the regular one has minimal surface area; see~\cite{Had57} (this also follows from~\cite{Pet61}). The literature contains very few other types of convex polytopes for which the problem has been understood. In~\cite{Fej48a,Fej64,FFK23} it was proved 
that for $\up=6,12$, among $3$-dimensional convex polytopes of given volume and having at most $\up$ facets, the ones with minimal surface area are the cube and dodecahedron, respectively. For convex polytopes in $\R^n$ that have $n+2$ vertices, the ones which minimize the isoperimetric quotient were determined in~\cite{BoB96}. For origin-symmetric polytopes in $\R^n$ that have $2n$ facets (parallelepipeds), \cite{Pet61} implies that hypercubes have minimal isoperimetric quotients. To the best of our knowledge, no other convex polytopes are known to be exact minimizers of the isoperimetric quotient while fixing their number of facets.

We were partially motivated to examine the question that Theorem~\ref{thm:isoperim for poly} answers by the recent work~\cite{NR23}, which constructs  for every $n\in \N$ a convex polytope $K\subset \R^n$ such that its integer translates $\{z+K\}_{z\in \Z^n}$ tile $\R^n$, yet $\vol_{n-1}(\partial K)=n^{1/2+o(1)}$ as $n\to \infty$; any measurable $E\subset \R^n$ such that $\{z+E\}_{z\in \Z^n}$ tile $\R^n$  must satisfy $\vol_n(E)=1$, so the aforementioned  statement from~\cite{NR23} is equivalent to $\iq(\partial K)=n^{1/2+o(1)}$. Beyond their intrinsic geometric interest, the search for such tiling bodies 
is motivated  by issues in computer science~\cite{FKO07,KORW08,KROW12}. The potential utility in this regard of the construction of~\cite{NR23} would necessitate the tiling body $K$ to have an ``efficient'' description, e.g. a constant-factor polynomial time optimization oracle (see~\cite{GLS93} for background) would be of use here (but, that is not all that would be needed for possible algorithmic implications). By Theorem~\ref{thm:isoperim for poly}, the number of facets of  $K$ must grow as $n\to \infty$ faster than any power of $n$, i.e., $K$ cannot have a short description as the intersection of $n^{O(1)}$ half-spaces. This rules out one possible approach to the aforementioned question, and it remains open to understand whether other routes towards obtaining an efficient version of a  tiling body as in~\cite{NR23} are possible.

\section{Proof of Theorem~\ref{thm:isoperim for poly} }\label{sec:proof1}

In the context of Theorem~\ref{thm:isoperim for poly}, we introduce the following notation for every $n,\up\in \N$ with $\up\ge n+1$: 
\begin{align}\label{eq:isop function notation}
\mathsf{Isoperim}_n(\up)\eqdef \inf \big\{\iq(K):\ K\subset \R^n\  \mathrm{is\ a\ convex\ polytope\ that\ has\ }\upphi\ \mathrm{facets}\big\}.
\end{align}
As discussed in Section~\ref{sec:history}, $\mathsf{Isoperim}_n(\up)$ has been computed for only  a few values of $n,\up$. The purpose of the present section is to prove Theorem~\ref{thm:isoperim for poly}, which evaluates $\mathsf{Isoperim}_n(\up)$ up to universal constant factors. Writing $\up=\varrho n$, using the notation~\eqref{eq:isop function notation} Theorem~\ref{thm:isoperim for poly} can be stated as follows:
$$
\forall n\in \N,\qquad  \mathsf{Isoperim}_n(\varrho n)\asymp \left\{\begin{array}{cl}n&\mathrm{if}\ 1+\frac{1}{n}\le \varrho\le 2,\\
\frac{n}{\sqrt{\log \varrho}}&\mathrm{if}\ 2\le \varrho\le 2^n,\\
\sqrt{n}&\mathrm{if}\ \varrho\ge 2^n.\end{array} \right.
$$

Remark~\ref{rem:approximate with more}  below records for ease of later references a standard observation that allows one to pass from isoperimetric statements about convex polytopes with at most a given number of facets or vertices to the corresponding statements for convex polytopes whose number of facets or vertices exactly equals  a given larger value; it implies in particular that   $\mathsf{Isoperim}_n(\upphi)\ge \mathsf{Isoperim}_n(\Phi)$ for $n,\up,\Phi\in \N$ with $\Phi\ge \up\ge n+1$. 

\begin{remark}\label{rem:approximate with more} {\em We will use the following very simple fact multiple times. Given $n,\up\in \N$, if $K\subset \R^n$ is a convex polytope with $\upphi$ facets (so, necessarily $\up\ge n+1$), then for every integer $\Phi\ge \upphi$ and every $\e>0$ there exists a convex polytope $L\subset K$ that has $\Phi$ facets and $\iq(K)-\e\le \iq(L)\le \iq(K)+\e$. One of multiple possible ways to  justify this  is to fix any vertex $v$ of $K$, by Hahn--Banach take $x^*\in \R^n$ and $\alpha\in \R$ such that $\langle x^*,v\rangle =\alpha$ and  $K\subset \{y\in\R^n:\ \langle x^*,y\rangle \le \alpha\}$, and for every $\d>0$ slice away from $K$ a small neighborhood of $v$ by considering the polytope $L(1,\d)=K\cap \{x\in \R^n:\ \langle x,v\rangle\le \alpha-\d\}$. Then, $\lim_{\d\to 0^+}\iq(L(1,\d))=\iq(K)$ and for small enough $\d$ the number of facets of $K$ equals $\upphi+1$. By iterating this procedure $\Phi-\upphi$ times we arrive at the desired polytope $L$. When $K$ is furthermore origin-symmetric,   $L$ can be taken to be  origin-symmetric by repeating the above with the  approximant $K\cap \{x\in \R^n:\ |\langle x,v\rangle|\le \alpha-\d\}$. Also,  if $K$ has $\ub\in \N$ vertices, then for every integer $\mathrm{B}\ge \ub$ and every $\e>0$ there is a convex polytope $L'\supseteq K$ that has $\mathrm{B}$ vertices and $\iq(K)-\e\le \iq(L')\le \iq(K)+\e$, as seen by taking any $u_1,\ldots,u_{\mathrm{B}-\ub}\in \R^n\setminus K$ that are in general position and sufficiently close to $K$, and considering $L'=\conv(K\cup\{u_1,\ldots,u_{\mathrm{B}-\ub}\})$. If $K$ is origin-symmetric and $\mathrm{B}$ is even, then $L'$ can be taken to be origin-symmetric, as seen  by considering $\conv(K\cup\{\pm u_1,\ldots,\pm u_{\mathrm{B}-\ub}\})$. 
}
\end{remark}

The following  simple lemma records  basic properties of the isoperimetric quotient: 

\begin{lemma}\label{lem:cones} Fix $n,\up\in \N$ and $h>0$.  Let $K\subset \R^n$ be a convex polytope whose facets are $F_1,\ldots,F_\upphi$. Suppose  that $K\supseteq hB^n$ and $F_1\cap hS^{n-1},\ldots,F_\upphi\cap hS^{n-1}\neq \emptyset$. Then
\begin{equation}\label{eq:iq for circumscribed}
\iq(K)=\frac{n}{h}\vol_n(K)^{\frac{1}{n}}.
\end{equation}
Furthermore, every convex body $K\subset \R^n$ for which $K\supseteq hB^n$ satisfies: 
\begin{equation}\label{eq:iq for including}
\iq(K)\le \frac{n}{h}\vol_n(K)^{\frac{1}{n}}.
\end{equation}
\end{lemma} 

\begin{proof} The identity~\eqref{eq:iq for circumscribed} is a consequence of the following obvious computation. By assumption, for every $i\in \{1,\ldots,\upphi\}$ there is $u_i\in hS^{n-1}$ such that $F_i\cap hS^{n-1}=\{u_i\}$. As $K\supseteq h B^n$, we necessarily have $F_i\subset u_i+u_i^\perp$. Hence, $\conv (\{0\}\cup F_i)$ is a cone whose cusp is at the origin $0\in \R^n$ and whose height equals $\|u_i\|_{\ell_{\!\!2}^n}=h$. The volume of $\conv (\{0\}\cup F_i)$ therefore equals $h\vol_{n-1}(F_i)/n$. As $\conv (\{0\}\cup F_1),\ldots,\conv (\{0\}\cup F_\up)$ have pairwise disjoint interiors and their union equals $K$, and also any two of the facets $F_1,\ldots,F_\up$ intersect in a set of codimension at least $2$ and their union equals $\partial K$, the desired identity~\eqref{eq:iq for circumscribed} is justified as follows: 
\begin{align*}
\vol_n(K)=\vol_{n} \Big(\bigcup_{i=1}^\upphi &\conv (\{0\}\cup F_i)\Big)=\sum_{i=1}^\upphi \vol_n\big(\conv (\{0\}\cup F_i)\big)\\&=\frac{h}{n}\sum_{i=1}^\upphi \vol_{n-1}(F_i)=\frac{h}{n}\vol_{n-1} \Big(\bigcup_{i=1}^\upphi F_i\Big)\stackrel{\eqref{eq:def iq}}{=}\frac{h}{n}\vol_{n-1}(\partial K)=\frac{h}{n}\iq(K)\vol_n(K)^{\frac{n-1}{n}}. 
\end{align*}

The general estimate~\eqref{eq:iq for including} is also known; see e.g.~\cite[Lemma~3]{NR23} for its proof. Very briefly,~\eqref{eq:iq for including} holds as:
\begin{equation}\label{convexity inclusion}
\forall s>0,\qquad K+sB^n\subset K+\frac{s}{h}K= \Big(1+\frac{s}{h}\Big)\bigg(\frac{1}{1+\frac{s}{h}}K+\Big(1-\frac{1}{1+\frac{s}{h}}\Big)K\bigg)=\Big(1+\frac{s}{h}\Big)K, 
\end{equation}
where the first inclusion in~\eqref{convexity inclusion} uses the assumption $K\subset h B^n$ and the last equality in~\eqref{convexity inclusion} is the crucial point where the assumed convexity of $K$  is used. The desired bound~\eqref{eq:iq for including} is now justified as follows:
\begin{equation*}
\iq(K)=\lim_{s\to 0^+} \frac{\vol_{n} (K+sB^n)-\vol_n(K)}{s\vol_n(K)^{\frac{n-1}{n}}}\stackrel{\eqref{convexity inclusion} }{\le} \limsup_{s\to 0^+} \frac{\left(1+\frac{s}{h}\right)^n-1}{s} \vol_n(K)^{\frac{1}{n}}=\frac{n}{h}\vol_n(K)^{\frac{1}{n}}.\tag*{\qedhere}
\end{equation*}
\end{proof}

For examples: the cross polytope $B_{\!\ell_{\!\!1}^n}\subset \R^n$ satisfies the assumption of Lemma~\ref{lem:cones} with $h=1/\sqrt{n}$ and its volume equals $2^n/n!$;  the  regular simplex with unit side length $\triangle_n\subset \R^n$ satisfies the assumption of Lemma~\ref{lem:cones} with $h=1/\sqrt{2n(n+1)}$ and its volume equals $2^{-n/2}\sqrt{n+1}/n!$.  Consequently, in these special cases~\eqref{eq:iq for circumscribed} becomes  the following more precise versions of the asymptotic behaviors of the isoperimetric quotients of the the cross-polytope and the regular simplex that were mentioned in the Introduction: 
\begin{equation}\label{eq:iq simplex and cross polytope}
\iq(\triangle_n)=\frac{n^\frac32(n+1)^{\frac12+\frac{1}{2n}}}{\sqrt[n]{n!}}= \big(e+o(1)\big)n\qquad\mathrm{and}\qquad  \iq(B_{\!\ell_{\!\! 1}^n})=\frac{2n^{\frac32}}{\sqrt[n]{n!}}=\big(2e+o(1)\big){\textstyle \sqrt{n}}.
\end{equation}

Prior to justifying the improved isoperimetric inequality~\eqref{eq:iq lower in thm} for polytopes, which is the more interesting part of Theorem~\ref{thm:isoperim for poly}, we will first quickly explain why the asymptotic lower bound on $\mathsf{Isoperim}_n(\up)$ in Theorem~\ref{thm:isoperim for poly} cannot be imporved for any $n,\up\in \N$ with $\up\ge n+1$. 

As the regular simplex in $\R^n$  has $n+1$ facets, $\mathsf{Isoperim}_n(\upphi)\le \iq(\Delta_n)\asymp n$  for every $\upphi\in \{n+1,n+2,\ldots\}$ by Remark~\ref{rem:approximate with more} and the first part of~\eqref{eq:iq simplex and cross polytope}. Thus,   $\mathsf{Isoperim}_n(\upphi)$ is at most a universal constant multiple of the right hand side of~\eqref{eq:iq lower in thm} for every $\up\in \N$ satisfying, say,  $n+1\le \upphi\le 3n$. As the  cross-polytope in $\R^n$ has $2^n$ facets, $\mathsf{Isoperim}_n(\upphi)\le \iq(B_{\!\ell_{\!\! 1}^n})\asymp \sqrt{n}$  for every integer $\upphi\ge 2^n$ by Remark~\ref{rem:approximate with more} and the second part of~\eqref{eq:iq simplex and cross polytope}. Hence, $\mathsf{Isoperim}_n(\upphi)$ is at most a universal constant multiple of the right hand side of~\eqref{eq:iq lower in thm}  also if $\upphi\ge 2^n$.

We may thus assume that $3n< \upphi< 2^n$. Since $\upphi/n> 3$, we can define  $m=m(\upphi,n)$ to be the largest element of $\{2,3,\ldots\}$ for which $\upphi/n> 2^m/m$. Then, $m\asymp \log(\upphi/n)$. Also $n\ge m+1$, as  $\{2^k/k: k\in \N\}$ is nondecreasing, so if  $m\ge n$, then 
$2^m/m\ge 2^n/n> \upphi/n$, where the last step uses the assumed upper bound on $\upphi$, in contradiction to our choice of $m$. So, we can write $n=am+r$ for some $a\in \N$ and $r\in \{1,\ldots,m\}$. 

Denote $b= \upphi -a2^m= \upphi -2^m(n-r)/m> n2^m/m-2^m(n-r)/m=r2^m/m\ge 2^r$, where the second step is a substitution of the definition of $a$, the third step uses  the definition of $m$, and the final step is holds because $1\le r\le m$ and $\{2^k/k: k\in \N\}$ is nondecreasing. By the above discussion, since $b\ge 2^r$ there exists a convex polytope $L\subset \R^r$ with $\vol_r(L)=1$ that has $b$ facets and satisfies $\iq(L)\asymp \sqrt{r}$.

We can now define a convex polytope $K\subset \R^n$ as follows: 
$$
K\eqdef \left(\frac{1}{\vol_m(B_{\!\ell_{\!\! 1}^m})^{\frac{1}{m}}}B_{\!\ell_{\!\! 1}^m}\right)^a\times L\subset (\R^m)^a\times \R^r\cong \R^{am+r}=\R^n. 
$$
The number of facets of $K$ equals $a2^m+b=\upphi$, by the definition of $b$ and the choice of $L$, and because the number of facets of $B_{\!\ell_{\!\! 1}^m}$ equals $2^m$.  We furthermore have $\vol_n(K)=1$ and 
$$
\iq(K)=\vol_{n-1}(\partial K)=a\iq(B_{\!\ell_{\!\! 1}^m})+\iq(L)\asymp a\sqrt{m}+\sqrt{r}\le 2\frac{am+r}{\sqrt{m}}= 2\frac{n}{\sqrt{m}}\asymp \frac{n}{\sqrt{\log \frac{\upphi}{n}}}.
$$

Having verified that there is a convex polytope with $\up$ facets for which the isoperimetric lower bound~\eqref{eq:iq lower in thm} of Theorem~\ref{thm:isoperim for poly}  cannot be improved, we will next explain why the isoperimetric quotient of every such $K$  must satisfy~\eqref{eq:iq lower in thm}. In the proof  we will use the following old theorem of Lindel\"of~\cite{Lin69}, which implies in particular that the minimum isoperimetric quotient among all the convex polytopes with a fixed number of facets is attained at a convex polytope  which is circumscribed around some Euclidean ball. Because Lindel\"of's theorem may not be very well-known, we will provide its proof below.

\begin{theorem}[Lindel\"of]\label{thm:lindelof}
\label{Lindelof} Fix $n,\upphi\in \N$ and let   $u_1,\ldots,u_\upphi\in S^{n-1}$ be distinct unit vectors that are not contained in any closed hemisphere of $S^{n-1}$, i.e., $\{u_1,\ldots,u_\upphi\} \not\subseteq  \{x\in \R^n: \langle x,v\rangle\le 0\}$ for every $v\in S^{n-1}$.  If $K\subset \R^n$ is a convex polytope such that the unit outer normal to each of its facets belongs to $\{u_1,\ldots,u_\upphi\}$, then 
\begin{equation}\label{eq:K0}\iq(K)\ge \iq (K_0),\qquad\mathrm{where}\qquad K_0=K_0(u_1,\ldots,u_\upphi)\eqdef \Big\{x\in \R^n:\ \max_{i\in \{1,\ldots,\upphi\}} \langle x,u_i\rangle \le 1\Big\}.
\end{equation}
\end{theorem}
\begin{proof} We are assuming that $u_1,\ldots,u_\upphi$ are not contained in any close hemisphere of $S^{n-1}$ to ensure that the convex polytope $K_0$ that is defined in~\eqref{eq:K0} is bounded.
The key (simple and elementary) observation is that the assumptions of Theorem~\ref{thm:lindelof} imply that there exist $\d_0=\d_0(K), C=C(K,u_1,\ldots,u_\upphi)>0$ such that: 
\begin{equation}\label{eq:parallel surface}
\forall 0\le \d\le  \d_0,\qquad \d\vol_{n-1}(\partial K)\le \vol_n\big((K+\d K_0)\setminus K\big)\le \d\vol_{n-1}(\partial K)+C\d^2.
\end{equation}

After~\eqref{eq:parallel surface} will be verified (below), Theorem~\ref{thm:lindelof} will quickly follow as $0\in K_0$, so $K+\d K_0\supseteq K$ for every $\d>0$, whence by the Brunn--Minkowski inequality~\cite{Bru89,Min03} (see e.g.~\cite{Bal97,Sch14}) the following estimate holds:
\begin{align}\label{eq:use BM}
\begin{split}
\vol_n\big((K+\d K_0)&\setminus K\big)=\vol_n(K+\d K_0)-\vol_n(K)\\&\ge \left(\vol_n(K)^{\frac{1}{n}}+\d\vol_n(K_0)^{\frac{1}{n}}\right)^n-\vol_n(K)=\Big(\vol_n(K)^{\frac{1}{n}}+\frac{\d}{n}\iq(K_0)\Big)^n-\vol_n(K),
\end{split}
\end{align}
where the last step of~\eqref{eq:use BM} is an instantiation of (the first part of) Lemma~\ref{lem:cones}, whose assumptions hold for $K_0$ with $h=1$. The derivation of conclusion~\eqref{eq:K0} of Theorem~\ref{thm:lindelof}  is now concluded as follows:
$$
\forall 0<\d\le \d_0,\qquad \iq(K)\stackrel{\eqref{eq:def iq}\wedge\eqref{eq:parallel surface}\wedge\eqref{eq:use BM} }{\ge} \frac{\Big(1+\frac{\d\iq(K_0)}{n\vol_n(K)^{\frac{1}{n}}}\Big)^n-1}{\d}\vol_n(K)^{\frac{1}{n}}-\frac{C\d}{\vol_n(K)^{\frac{n-1}{n}}}\xrightarrow[\d\to 0]{} \iq(K_0).  
$$

It remains to explain why~\eqref{eq:parallel surface} holds; we used only the second inequality in~\eqref{eq:parallel surface} to prove Theorem~\eqref{thm:lindelof}, but both of the inequalities in~\eqref{eq:parallel surface} are quick to verify. The assumption on $K$ means that there exists a subset  $I$ of $\{1,\ldots,\upphi\}$ and for every $i\in I$ there are $\ut_i\in \R$ such that 
\begin{equation}\label{eq: K as intersection}
K=\bigcap_{i\in I} \big\{x\in \R^n:\ \langle x,u_i\rangle \le \ut_i\big\}.\end{equation}
Furthermore, there is $\d_0=\d_0(K)>0$ such that if $i,j\in I$ satisfy $u_j=-u_i$, then $\ut_i+\ut_j\ge \d_0$, since otherwise  the right hand side of~\eqref{eq: K as intersection} would be contained in the hyperplane $\{x\in \R^n:\ \langle x,u_i\rangle=\tau_i\}$. 

For each $i\in I$, let $F_i\subset \partial K$ denote the facet of $K$  whose unit outer normal is $u_i$. Thus, 
\begin{equation}\label{eq:facet of K}
F_i=\big\{x\in \R^n:\ \langle x,u_i\rangle = \ut_i\big\}\bigcap \left(\bigcap_{j\in I\setminus \{i\}} \big\{x\in \R^n:\ \langle x,u_j\rangle \le \ut_j\big\}\right).
\end{equation}
Since the definition of $K_0$ in~\eqref{eq:K0} implies that $[0,1]u_i\subset K_0$ for every $i\in I$, we have the following inclusion: 
\begin{equation}\label{eq:include face slabs}
\forall \d\ge 0, \qquad (K+\d K_0)\setminus K\supseteq \bigcup_{i\in I} \big(F_i+[0,\d]u_i\big). 
\end{equation}
For $\d \ge 0$ and $i\in I$,  $\vol_n(F_i+[0,\d]u_i)=\d\vol_{n-1}(F_i)$, as  $u_i$ is normal to $F_i$. The interiors of $\{F_i+[0,\d]u_i\}_{i\in I}$ are disjoint and $\sum_{i\in I} \vol_{n-1}(F_i)=\vol_{n-1}(\partial K)$, so the first inequality in~\eqref{eq:parallel surface} follows from~\eqref{eq:include face slabs}.

To verify the second inequality in~\eqref{eq:parallel surface},  for every $i,j\in I$ such that $u_j\notin \{u_i,-u_j\}$, whence $u_i,u_j$ are linearly independent and $-1<\langle u_i,u_j\rangle<1$, let $W_{ij}\subset \R^n$ denote the span of $u_i,u_j$.  Also, let $w_{ij}$ be the unique vector in $W_{ij}$ satisfying $\langle w_{ij},u_i\rangle=\ut_i$ and $\langle w_{ij},u_j\rangle=\ut_j$.\footnote{Explicitly, $w_{ij}=\frac{\ut_i-\ut_j\langle u_i,u_j\rangle}{1-\langle u_i,u_j\rangle^2}u_i+\frac{\ut_j-\ut_i\langle u_i,u_j\rangle}{1-\langle u_i,u_j\rangle^2}u_j$, though an exact expression for $w_{ij}$ is not needed for the reasoning herein.} 

With these notations, the following counterpart to~\eqref{eq:include face slabs} holds:  
\begin{equation}\label{eq:reverse inclusion}
\forall 0\le \d\le \d_0,\qquad (K+\d K_0)\setminus K\setminus \bigcup_{i\in I} \big(F_i+[0,\d]u_i\big) \subset \bigcup_{\substack{i,j\in I\\ u_j\neq\pm u_i}}\Big(W_{ij}^\perp+ w_{ij}+\frac{2\d}{1-|\langle u_i,u_j\rangle|} W_{ij}\cap B^n\Big).
\end{equation}
After~\eqref{eq:reverse inclusion} will be checked (below), the second inequality in~\eqref{eq:parallel surface} would be deduced as follows: 
\begin{align}\label{eq:sum over codim 2}
\begin{split}
\vol_n\big((K+\d K_0)\setminus K\big)&\stackrel{\eqref{eq:reverse inclusion}}{\le} \d\sum_{i\in I} \vol_{n-1}(F_i)+\sum_{\substack{i,j\in I\\ u_j\neq\pm u_i}} \vol_{n-2} \big(W_{ij}^\perp\cap (K+\d_0 K_0)\big)\frac{4\d^2}{(1-|\langle u_i,u_j\rangle|)^2}\vol_2(B^2)\\ &\ =\d\vol_{n-1}(\partial K)+4\pi\bigg(\sum_{\substack{i,j\in I\\ u_j\neq\pm u_i}} \frac{\vol_{n-2} \big(W_{ij}^\perp\cap (K+\d_0 K_0)\big)}{(1-|\langle u_i,u_j\rangle|)^2}\bigg)\d^2,
\end{split} 
\end{align}
which indeed gives~\eqref{eq:parallel surface} as $K+\d_0K_0$ is bounded, so each of the summands appearing in~\eqref{eq:sum over codim 2}  is finite. 

It remains to verify~\eqref{eq:reverse inclusion}. Suppose that $0<\d\le \d_0$ and  $x$ belongs to the left hand side of~\eqref{eq:reverse inclusion}. Then $x\notin K$, so by the representation of $K$ in~\eqref{eq: K as intersection}  there exists $i\in I$ such that $\langle x,u_i\rangle>\ut_i$. Also,  $x\in K+\d K_0$, so by the representation of $K$ in~\eqref{eq: K as intersection} and the definition of $K_0$ in~\eqref{eq:K0} we know that $\langle x,u_j\rangle \le \ut_j+\d$ for every $j\in I$. Thus, writing $\upsigma =\langle x,u_i\rangle-\ut_i$, we have $0<\upsigma\le \d$. Thanks to the assumed  membership of $x$ in the left hand side of~\eqref{eq:reverse inclusion}, necessarily  $x-\upsigma u_i\notin F_i$. But $\langle x-\upsigma u_i,u_i\rangle =\ut_i$ by the definition of $\upsigma$, so by the representation~\eqref{eq:facet of K} of $F_i$ this entails that there is $j\in I\setminus \{i\}$ for which $\langle x-\upsigma  u_i,u_j\rangle>\ut_j$. In summary, we have $\ut_j- \d\le \ut_j-\upsigma\langle u_i,u_j\rangle < \langle  x,u_j\rangle \le \ut_j+\d$ and $\ut_i<\langle x,u_i\rangle \le \ut_i+\d$. This implies that $u_j\neq u_i$ because otherwise it would follow that $\ut_i<\langle x,u_i\rangle <-\ut_j+\d_0\le -\ut_j+\d_0$, in contradiction to the definition of $\d_0$. Hence, the subspace $W_{ij}$ is defined per the notation that was introduced above. Let $y$ be the orthogonal projection of $x$ to $W_{ij}$. Then, $x-y$ is perpendicular to  $u_i$ and $u_j$, so $\langle y-w_{ij},u_i\rangle=\langle x,u_i\rangle -\ut_i\in   (0,\d]$ and $\langle y-w_{ij},u_j\rangle=\langle x,u_i\rangle -\ut_j\in   (-\d,\d]$, as by definition $\langle w_{ij},u_i\rangle=\ut_i$ and $\langle w_{ij},u_j\rangle=\ut_j$.  It is straightforward to compute that every $v\in W_{ij}$ satisfies $$
\|v\|_{\!\ell_{\!\!2}^n}=\sqrt{\frac{\langle v,u_i\rangle^2+\langle v,u_i\rangle^2-2\langle u_i,u_j\rangle\langle v,u_i\rangle\langle v,u_j\rangle} {1-\langle u_i,u_j\rangle^2}}.
$$
Since $y,w_{ij}\in W_{ij}$, we may apply this identity to $y-w_{ij}$ in combination with the aforementioned bounds on $\langle v,u_i\rangle,\langle v,u_j\rangle$ to deduce that the $\ell_{\!\!2}^n$ distance between $y$ and $w_{ij}$ is at most $2\d^2/(1-|\langle u_i,u_j\rangle|)$. As $y$ is the orthogonal projection of $x$ onto $W_{ij}$, this means that $x\in W_{ij}^\perp+w_{ij}+(2\d^2/(1-|\langle u_i,u_j\rangle|))B^n\cap W_{ij}$.\end{proof}

We can now complete the derivation of~\eqref{eq:iq lower in thm}, which is the main part of Theorem~\ref{thm:isoperim for poly}:  

\begin{proof}[Proof of~\eqref{eq:iq lower in thm}] By~\cite{CP88,Glu88} (see also the exposition in~\cite[Theorem~8]{Bal01}), every $u_1,\ldots,u_\up\in S^{n-1}$ satisfy:   
\begin{equation}\label{eq:quote CP}
\vol_n\big(\{x\in \R^n:\ \max_{i\in \{1,\ldots,\up\}} |\langle x,u_i\rangle|\le 1\}\big)^{\frac{1}{n}}\gtrsim \frac{1}{\sqrt{1+\log \frac{\up}{n}}}.
\end{equation}
If  $K\subset \R^n$ is a convex polytope that has $\up$ facets, then let $u_1,\ldots,u_\up\in S^{n-1}$ be the unit outer normals to its facets. By combining Theorem~\ref{thm:lindelof} with Lemma~\ref{lem:cones} it follows that:
\begin{align*}
\iq(K)\ge \iq \big(\{x\in \R^n:\ \max_{i\in \{1,\ldots, \up\}} \langle x,u_i\rangle\le 1\}\big)&\stackrel{\eqref{eq:iq for circumscribed}}{=}n\vol_n\big(\{x\in \R^n:\ \max_{i\in \{1,\ldots,\up\}} \langle x,u_i\rangle\le 1\}\big)^{\frac{1}{n}}\\&\ge n\vol_n\big(\{x\in \R^n:\ \max_{i\in \{1,\ldots,\up\}} |\langle x,u_i\rangle|\le 1\}\big)^{\frac{1}{n}}\stackrel{\eqref{eq:quote CP}}{\gtrsim} \frac{n}{\sqrt{1+\log \frac{\up}{n}}}.
\end{align*}
The remaining part $\iq(K)\gtrsim \sqrt{n}$ of~\eqref{eq:iq lower in thm} is a special case of the ``vanilla'' isoperimetric theorem~\eqref{eq:quote isoperimetric theorem}.  
\end{proof}

\begin{remark} {\em Theorem~\ref{thm:isoperim for poly} can be used to slightly streamline  the proof that for every $n\in \N$ there is a convex body $K\subset \R^n$ of volume $1$ such that the  area of its orthogonal projection onto {\em every} hyperplane is at least a positive universal constant multiple of $\sqrt{n}$. The existence of this pathological body is due to~\cite{Bal91d}; we will next justify why it holds using the same principles as in the reasoning of~\cite{Bal91d}, except that the calculation in its punchline can  now be done automatically by appealing to (the case $\up=4n$ of) Theorem~\ref{thm:isoperim for poly}. 

By~\cite{FLM77,Kasref} there are $v_1,\ldots,v_{2n}\in S^{n-1}$ such that $\sum_{i=1}^{2n} |\langle \theta,v_i\rangle| \gtrsim \sqrt{n}$ for every $\theta\in S^{n-1}$. By Minkowski's existence theorem~\cite{Min03}, there is an origin-symmetric convex polytope $K\subset \R^n$ with  $\vol_n(K)=1$ that has $\up=4n$ facets  $\pm F_1,\ldots,\pm F_{2n}$ satisfying  $\vol_{n-1}({F_1})=\ldots=\vol_{n-1}({F_{2n}})$, and  for every $i\in \{1,\ldots,2n\}$ the unit outer normal to $F_i$ equals $v_i$. Given $\theta\in S^{n-1}$,  the $(n-1)$-dimensional volume of the orthogonal projection of $K$ onto the hyperplane $\theta^\perp$ equals $\sum_{i=1}^{2n}\vol_{n-1}(F_i)|\langle \theta,v_i\rangle|$, by e.g.~\cite[equation~(13.12)]{San04}. By assumption, $\vol_{n-1}(F_i)=\vol_{n-1}(\partial K)/\up\asymp \vol_{n-1}(\partial K)/n$ for every $i\in \{1,\ldots,2n\}$, so the choice of $v_1,\ldots,v_{2n}$ ensures that the area of  the orthogonal projection of $K$  onto $\theta^\perp$ is at least a positive universal constant multiple of $\vol_{n-1}(\partial K)/\sqrt{n}$. Finally, as $K$ has $O(n)$ facets and unit volume, $\vol_{n-1}(\partial K)\gtrsim n$ by Theorem~\ref{thm:isoperim for poly}. 
}
\end{remark}

\section{proof of Theorem~\ref{thm:vertex reverse is}}\label{sec:proof2}

Given $n\in \N$ and  a convex body $K\subset \R^n$, let $\sigma_{K}$ denote the area measure  of $K$ (see e.g.~\cite[Section~10.1]{Gru07}), which is the Gauss map-pullback to $S^{n-1}$ of the restriction to $\partial K$ of the $(n-1)$-dimensional Hausdorff measure induced by the $\ell_{\!\!2}^n$ metric. Specifically,  for a Borel subset $E$ of $S^{n-1}$, one defines $\sigma_{K}(E)$ to be the $\vol_{n-1}$-measure  of the subset of $\partial K$ that consists of all those $x\in \partial K$ for which there is a unit outer normal to $\partial K$ at $x$ that belongs to $E$.  Thus, $\sigma_K(S^{n-1})=\vol_{n-1}(\partial K)$. We also note for later reference the following straightforward change of variable identity (see~\cite{Pet61} or e.g.~equation~(2.1) in~\cite{GP99}):  
\begin{equation}\label{eq:chage of variable}
\forall \T\in \GL_n(\R),\qquad \vol_{n-1}(\partial \T K)=\int_{S^{n-1}} \|(\T^*)^{-1}u\|_{\!\ell_{\!\!2}^n}\ud \sigma_K(u).  
\end{equation}

By~\cite{Pet61}  the minimum in~\eqref{eq:state reverse iso} exists and the corresponding minimizing matrix is unique up to orthogonal transformations, i.e.,  if $\A,\B\in \SL_n(K)$ are such that $\iq(\A K) =\iq(\B K)=\partial_K$, then necessarily $\A\B^{-1}\in \O_n$. It was proved in~\cite{Pet61} (see also~\cite[Theorem~1]{GP99}) that  $\iq(\A K)=\partial_K$ for some $\A\in \SL_n(\R)$  if and only if the covariance matrix  $(\int_{S^{n-1}} u_iu_j\ud \sigma_{\A K}(u))_{(i,j)\in \n\times \n}\in \M_n(\R)$ of $\sigma_{\A K}$ is a scalar multiple of the identity. That scalar must equal  $\vol_{n-1}(\partial \A K)/n=\iq(\A K)\vol_n(K)^{(n-1)/n}/n$, as seen by comparing traces and using the fact that $\vol_n(\A K)=\vol_n(K)$, since $\A\in \SL_n(\R)$. Hence, the following holds for every $\A\in \SL_n(\R)$:
\begin{equation}\label{isotropic requirement}
 \iq(\A K)=\partial_K\iff \forall \B\in \M_n(\R),\qquad \int_{S^{n-1}} \langle u,\B u\rangle \ud \sigma_{\A K}(u)=\frac{\iq(\A K)\vol_n(K)^{\frac{n-1}{n}}}{n}\mathrm{Trace}(\B). 
\end{equation}

The following lemma  gives an a priori estimate on the Schatten--von Neumann-$1$ norm\footnote{Given $n\in \N$ and $\sC\in \M_n(\R)$, the Schatten--von Neumann-$1$ norm~\cite{vNe37} of $\sC$, denoted $\|\sC\|_{\S_1^n}$, is the trace of $(\sC^*\sC)^{\frac12}$.}   of the surface area  minimizer in~\eqref{eq:state reverse iso} that facilitates a subsequent compactness argument:

\begin{lemma}\label{prop:minimum} For every $n\in \N$ and every convex body $K\subset \R^n$, if $\A\in \SL_n(K)$ satisfies $\iq(\A K) =\partial_K$, then: 
\begin{equation}\label{eq:Schatten 1 appriori}
\|\A\|_{\S_1^n}\lesssim \sqrt{n}\iq(K). 
\end{equation}
\end{lemma}

\begin{proof} As $\iq(K) =\iq((\A^*\A)^{-\frac12}\A K)$, since $(\A^*\A)^{-\frac12}\A\in \O_n$,  we have:
\begin{align*}
\iq(K) =\iq\big((\A^*\A)^{-\frac12}&\A K\big)\stackrel{\eqref{eq:chage of variable}}{=}\frac{1}{\vol_{n}(K)^{\frac{n-1}{n}}}\int_{S^{n-1}}\big\|(\A^*\A)^{\frac12} u\big\|_{\!\ell_{\!\!2}^n}\ud\sigma_{\A K}(u)\\ &\ge \frac{1}{\vol_{n}(K)^{\frac{n-1}{n}}}\int_{S^{n-1}} \langle u,(\A^*\A)^{\frac12}u\rangle \ud\sigma_{\A K}(u)\stackrel{\eqref{isotropic requirement}}{=}\frac{\iq(\A K)}{n}\|\A\|_{\S_1^n}\stackrel{\eqref{eq:quote isoperimetric theorem}}{\gtrsim} \frac{1}{\sqrt{n}}\|\A\|_{\S_1^n}.\tag*{\qedhere}
\end{align*}
\end{proof}

The compactness statement that we alluded to above is the following:

\begin{lemma}\label{lem:precomapctness partial K} Fix $n\in \N$. Let $\{K_m\}_{m=1}^\infty$ be a sequence of convex bodies in $\R^n$ satisfying $\sup_{m\in \N} \vol_n(K_m)<\infty$, and furthermore there is $r >0$ such that $K_m\supseteq r B^n$ for every $m\in \N$. Then, there is a subsequence $\{K_{m_i}\}_{i=1}^\infty$ and a convex body $K_\infty\subset \R^n$ with $\lim_{i\to \infty}K_{m_i}=K_\infty$ (in the Hausdorff metric) and $\lim_{i\to \infty} \partial_{K_{m_i}}=\partial_{K_\infty}$. 
\end{lemma}

\begin{proof} Write $V=\sup_{m\in \N} \vol_n(K_m)$. Then, $K_1,K_2,\ldots \subset R B^n$ for  $R=nV/(r^{n-1}\vol_{n-1}(B^{n-1}))$ since if $m\in \N$ and there were $x\in K_m$ with $\|x\|_{\!\ell_{\!\!2}^n}>R$, then as $K_m\supseteq R B^n\supseteq x^\perp\cap R B^n$, by convexity $K_m$ would contain the cone $\conv (\{x\}\cup (x^\perp\cap RB^{n}))$ whose base is $x^\perp\cap rB^n$ and whose height is $\|x\|_{\!\ell_{\!\!2}^n}$, yielding the contradiction $V\ge \vol_n(K_m)\ge \vol_n(\conv (\{x\}\cup (x^\perp\cap rB^n)))=\|x\|_{\!\ell_{\!\!2}^n}\vol_{n-1}(rB^{n-1})/n>Rr^{n-1}\vol_{n-1}(B^{n-1})/n=V$.  Thus,  all of $\{K_m\}_{m=1}^\infty$ are contained in a compact set, so there exists a subsequence $\{K_{m(k)}\}_{k=1}^\infty$ and a convex set $K_\infty\subset \R^n$ such that $\lim_{k\to \infty}K_{m(k)}=K_\infty$ in the Hausdorff metric. As $K_m\supseteq rB^n$ for every $m\in \N$, also $K_\infty\supseteq rB^n$, so $K_\infty$ is a convex body. The Hausdorff convergence thus implies   $\lim_{k\to \infty} \iq(K_{m(k)})=\iq(K_\infty)$.

For each $m\in \N$ fix a matrix $\A_m\in \SL_n(\R)$ such that $\iq(\A_m)=\partial_{K_m}$.  By (the second part of) Lemma~\ref{lem:cones} we have $\sup_{m\in \N} \iq(K_m)\le nV^{1/n}/r$. Thanks to Lemma~\ref{prop:minimum} we therefore have $\sup_{m\in \N} \|\A_m\|_{\S_1^n}\lesssim n^{3/2}V^{1/n}/r$. So, even though $\SL_n(\R)$ is not compact, all of the matrices $\{\A_m\}_{m=1}^\infty$ belong to a compact subset, whence there exists a subsequence $\{m_i\}_{i=1}^\infty$ of $\{m(k)\}_{k=1}^\infty$ such that $\lim_{i\to \infty} \A_{m_i}=\A_\infty$ for some  $\A_\infty\in \SL_n(\R)$.  Every $\B\in \SL_n(\R)$ satisfies $\iq(\B\A_\infty K_\infty)=\lim_{i\to \infty} \iq(\B\A_{m_i}K_{m_i})\ge \lim_{i\to \infty}  \iq (\A_{m_i}K_{m_i})=\iq(A_\infty K_\infty)$, where the inequality holds by the choice of $\A_{m_i}$. Hence, $\partial_{ K_\infty}=\iq(\A_\infty K_\infty)=\lim_{i\to \infty}  \iq (\A_{m_i}K_{m_i})=\lim_{i\to \infty}  \partial_{K_{m_i}}$. 
\end{proof}

The regular simplex $\triangle_n$ has $n+1$ vertices and $n+1$ facets, and it satisfies $\partial_{\triangle_n}=\iq(\triangle_n)\asymp n$, using the aforementioned isotropicity criterion~\eqref{isotropic requirement} of~\cite{Pet61} and~\eqref{eq:iq simplex and cross polytope}.  By combining Remark~\ref{rem:approximate with more} and Lemma~\ref{lem:precomapctness partial K}, it follows that for every $\up,\ub\ge n+1$ there exist convex polytopes $K_\ub,K_\up\subset \R^n$ such that $K_\up$ has (exactly) $\up$ facets, $K_\ub$ has $\ub$ vertices, and $\partial_{K_\up}\asymp n\asymp\partial_{K_\ub}$. In the same vein, the hypercube $[-1,1]^n$ is origin symmetric, has $2n$ facets, and satisfies $\partial_{[-1,1]^n}=\iq([-1,1]^n)=2n$, so by combining Remark~\ref{rem:approximate with more} and Lemma~\ref{lem:precomapctness partial K} it follows that for every $\up\ge 2n$ there exists convex polytope $K\subset \R^n$ that has $\up$ facets and satisfies $\partial_{K}\asymp n$. This completes the justification of the statements that were made in the Introduction to explain why Theorem~\ref{thm:vertex reverse is} treats only the case of origin-symmetric polytopes that have a fixed number of vertices. 

The following lemma presents  volume and surface area computations that will be used later to justify why the estimate~\eqref{eq:reverse isoperim in thm} of Theorem~\ref{eq:reverse isoperim in thm} is sharp:
\begin{lemma}\label{lem:l1 sum 1} Fix $a, b_1,\ldots,b_a,\ub_1,\ldots,\ub_a,\up_1,\ldots,\up_a\in \N$. For each $i\in \{1,\ldots,a\}$, let $K_i\subset \R^{b_i}$ be a convex polytope that has $\ub_i$ vertices and $\up_i$ facets, such that $K_i\supseteq h_i B^{b_i}$ for some $h_i>0$ and every facet of $K_i$ 
has nonempty intersection with $h_iB^{b_i}$ (i.e., the assumption of Lemma~\ref{lem:cones}  holds for $K_1,\ldots,K_a$, in their respective dimensions). Define $K=K(K_1,\ldots,K_a)\subset \R^{b_1}\times \ldots\times \R^{b_a}\cong \R^{b_1+\ldots+b_a}$ as follows: 
\begin{equation}\label{eq:def ell_1 sum}
K\eqdef \Big\{(\lambda_1 x_1,\ldots,\lambda_a x_a):\ (x_1,\ldots,x_a)\in K_1\times \ldots\times K_a\quad\mathrm{and}\quad  \ \lambda_1\ldots,\lambda_a\in [0,1]^a\ \quad\mathrm{and}  \quad \sum_{i=1}^a\lambda_i=1\Big\}.
\end{equation} 
Then, $K$ is a convex polytope that has $\ub_1+\ldots+\ub_a$ vertices and $\up_1\cdots\up_a$ facets whose volume is given by:
\begin{equation}\label{volume formula l1 sum}
\vol_{b_1+\ldots+b_a}(K)= \frac{1}{(b_1+\ldots+b_a)!}\prod_{i=1}^a b_i!\vol_{b_i}(K_i), 
\end{equation}
and whose surface area is given by:
\begin{equation}\label{eq:surface area forluma l1 sum}
\vol_{b_1+\ldots+b_a-1}(\partial K)= \frac{\sqrt{\frac{1}{h_1^2}+\ldots+\frac{1}{h_a^2}}}{(b_1+\ldots+b_a-1)!}\prod_{i=1}^a b_i!\vol_{b_i}(K_i). 
\end{equation}

Suppose furthermore that $K_1,\ldots, K_a$ are origin-symmetric and that for every $i\in \{1,\ldots,a\}$ the facets of $K_i$ are congruent to each other, i.e., if $F,F'$ are facets of $K_i$ then there exists an isometry $\mathsf{J}=\mathsf{J}_{F,F'}:\R^{b_i}\to \R^{b_i}$ such that $\mathsf{J}(F)=F'$. If also for every $i\in \{1,\ldots,a\}$ we have  $\partial_{K_i}=\iq(K_i)$ and $h_i^2=1/\sqrt{b_i}$, then:
\begin{equation}\label{minimum surface area formula}
\partial_K=\iq(K)=\Bigg(\frac{\prod_{i=1}^a b_i^{-\frac32 b_i}b_i!}{(b_1+\ldots+b_a)!}\Bigg)^{\frac{1}{b_1+\ldots+b_a}}(b_1+\ldots+b_a)^\frac32 \prod_{i=1}^a \partial_{K_i}^{\frac{b_i}{b_1+\ldots+b_a}}\asymp \sqrt{b_1+\ldots+b_a}\prod_{i=1}^a \bigg(\frac{1}{\sqrt{b_i}}\partial_{K_i}\bigg)^{\frac{b_i}{b_1+\ldots+b_a}}. 
\end{equation}
\end{lemma}

Prior to proving Lemma~\ref{lem:l1 sum 1}, we will use it to justify the second part of Theorem~\ref{eq:reverse isoperim in thm}, namely, the optimality   of~\eqref{eq:reverse isoperim in thm}. Fix $n\in \N$ and $\ub\in 2\N$ satisfying $\ub\ge 2n$. If $\ub\ge 2^n$, then as $[-1,1]^n$ has $2^n$ vertices, by combining Remark~\ref{rem:approximate with more} and Lemma~\ref{lem:precomapctness partial K}  we see that there is an origin-symmetric convex polytope $K\subset \R^n$ that has $\ub$ vertices and $\partial_K\gtrsim \partial_{[-1,1]^n}=2n$, so indeed~\eqref{eq:reverse isoperim in thm} is sharp in this case. We may therefore assume from now that $2n\le \ub<2^n$. Let $m$ be the largest element of $\{2,3,\ldots\}$ such that $2^m/m\le \ub/n$.  Then $m\asymp \log(\ub/n)$. Also, $m\ge n+1$ because otherwise $2^m/m\le \ub/n<2^n/n$ in contradiction to our assumption on $\ub$. We can therefore divide with remainder to write $n=(a-1)m+r$ for some integer $a\ge 2$ and some  $r\in \{1,\ldots,m\}$. 

Apply Lemma~\ref{lem:l1 sum 1} to $K_1=\ldots=K_{a-1}=(1/\sqrt{m})[-1,1]^m$ and $K_{a}=(1/\sqrt{r})[-1,1]^n$. Thus, in the notations of Lemma~\ref{lem:l1 sum 1}, we have  $b_1=\ldots=b_{a-1}=m$ and $b_a=r$, and also $h_1=\ldots=h_{a-1}=1/\sqrt{m}$ and $h_a=1/\sqrt{r}$. As all of the assumptions of Lemma~\ref{lem:l1 sum 1} that ensure that~\eqref{minimum surface area formula} holds are satisfied, we obtain a convex polytope $K$  in $\R^n$ that has $(a-1)2^m+2^r$ vertices for which we have the following estimate:
$$
\partial_K\asymp \sqrt{n}\big(2\sqrt{m}\big)^{\frac{(a-1)m}{n}}\big(2\sqrt{r}\big)^{\frac{r}{n}}=2\sqrt{nm}\left(\frac{r}{m}\right)^{\frac{r}{2n}}\ge 2\sqrt{nm}\left(\frac{r}{m}\right)^{\frac{r}{2m}}\asymp \sqrt{nm}\asymp {\textstyle \sqrt{n\log\frac{\ub}{n}}}.
$$
Observe that $(a-1)2^m+2^r=(n-r)2^m/m+2^r=n2^m/m-r(2^m/m-2^r/r)\le n2^m/m\le \ub$, where the penultimate inequality holds as $r\in \{1,\ldots,m\}$ and final inequality holds by the definition of $m$. So, the number of vertices of $K$ is at most $\ub$.  By  combining Remark~\ref{rem:approximate with more} and Lemma~\ref{lem:precomapctness partial K} we conclude that there exists an origin-symemtric convex polytope $K'\subset \R^n$ that has (exactly) $\ub$ vertices and satisfies $\partial_{K'}\asymp \sqrt{n\log (\ub/n)}$.

\begin{proof}[Proof of Lemma~\ref{lem:l1 sum 1}] For every $i\in \{1,\ldots,a\}$ let $v_{i1},\ldots,v_{i,\ub_i}$ be the vertices of $K_i$ and let $F_{i,1},\ldots,F_{i,\up_i}$ be the facets of $K_i$. Denote the canonical copy of $K_i$ in $\R^{b_1}\times \ldots\times \R^{b_a}$ by $K_i'$, i.e., $K_i'$ consists  of those $(x_1,\ldots,x_a)$ for which $x_i\in K_i$  and $x_j=0$ for every $j\in \{1,\ldots,a\}\setminus\{i\}$. The body $K$ in~\eqref{eq:def ell_1 sum} is the convex hull of $K_1'\cup\ldots\cup K_a'$. 

Thus, $K$   is a convex polytope whose vertices are the $\ub_1+\ldots+\ub_a$ vectors $(x_1,\ldots,x_a)\in \R^{b_1}\times \ldots\times \R^{b_a}$  for which there is $i\in \{1,\ldots,a\}$ such that $x_i=v_{ij}$ for some $j\in \{1,\ldots,\ub_i\}$, and  $x_k=0$ for every $k\in \{1,\ldots,a\}\setminus\{i\}$. The facets of $K$ are the $\up_1\cdots\up_a$ sets  $\{F_{j_1\ldots j_a}: (j_1,\ldots,j_a)\in \{1,\ldots,\up_1\}\times \ldots\times \{1,\ldots,\up_a\}\}$, where for every  $(j_1,\ldots,j_a)\in \{1,\ldots,\up_1\}\times \ldots\times \{1,\ldots,\up_a\}$ we introduce the following notation: 
\begin{equation}\label{eq:facet formula}
F_{j_1\ldots j_a}\!\!\eqdef \! \Big\{(\lambda_1 x_1,\ldots,\lambda_a x_a):\ (x_1,\ldots,x_a)\in F_{1j_1}\times \ldots\times F_{a j_a}\quad\mathrm{and}\quad  \ \lambda_1\ldots,\lambda_a\in [0,1]^a\ \quad\mathrm{and}  \quad \sum_{i=1}^a\lambda_i=1\Big\}. 
\end{equation}

Expression~\eqref{volume formula l1 sum} for the volume of $K$ is a direct consequence of its definition via the following simple induction. If $a=1$, then $K=[0,1]K_1=K_1$ as $0\in K_1$ and $K_1$ is convex, so both sides of~\eqref{volume formula l1 sum} equal $\vol_{b_1}(K_1)$.   If $a>1$, then let $L\subset \R^{b_1}\times \ldots\times \R^{b_{a-1}}$ be the body obtained by the above procedure for $K_1,\ldots, K_{a-1}$, namely:
$$
L= \Big\{(\lambda_1 x_1,\ldots,\lambda_{a-1} x_{a-1}):\ (x_1,\ldots,x_{a-1})\in K_1\times \ldots\times K_{a-1}\quad\mathrm{and}\quad  \ \lambda_1\ldots,\lambda_{a-1}\in [0,1]^{a-1}\ \quad\mathrm{and}  \quad \sum_{i=1}^{a-1}\lambda_i=1\Big\}.
$$
The orthogonal projection of $K$ onto $\R^{b_a}$ equals $[0,1]K_a=K_a$, as $0\in K_a$ and $K_a$ is convex. Denote the Minkowski functional of $K_a$ by $\|\cdot\|_{K_a}:\R^{b_a}\to [0,\infty)$, i.e.,  $\|y\|_{K_a}=\inf\{s>0:\ (1/s)y\in K_a\}$ for  $y\in\R^{b_a}$. Then:
$$
\forall y\in K_a,\qquad \Big\{(\lambda_1 x_1,\ldots,\lambda_{a-1} x_{a-1})\in \R^{b_1}\times \ldots\times \R^{b_{a-1}}:\ (\lambda_1 x_1,\ldots,\lambda_{a-1} x_{a-1},y)\in K\Big\}
=\big(1-\|y\|_{K_a}\big)L. 
$$
By Fubini we therefore see that:
\begin{equation}\label{eq:beta function identity}
\frac{\vol_{b_1+\ldots+b_a}(K)}{\vol_{b_1+\ldots+b_{a-1}}(L)}=\int_{K_a} \big(1-\|y\|_{K_a}\big)^{b_1+\ldots+b_{a-1}}\ud y= \frac{b_a!(b_1+\ldots+b_{a-1})!}{(b_1+\ldots+b_{a})!}\vol_{b_a}(K_a),
\end{equation}
where the last step of~\eqref{eq:beta function identity} is standard (e.g., by substituting the function $(y\in \R^{b_a})\mapsto (1-\|y\|_{K_a})^{b_1+\ldots+b_{a-1}}$ into~\cite[Proposition~1]{NR03}, which is integration in polar coordinates with respect to the cone measure~\cite{GM87} of $K_a$). This establishes the inductive step for~\eqref{volume formula l1 sum}, we will next proceed to justify~\eqref{eq:surface area forluma l1 sum}.   

Fix $x=(x_1,\ldots,x_a)\in \R^{b_1}\times \ldots\times \R^{b_a}$ such that $x_i\neq 0$ for all $i\in \{1,\ldots,a\}$. Let $d_i$ denote the $\ell_{\!\!2}^{b_i}$-norm of $x_i$. If the Euclidean norm $d=(d_1^2+\ldots+d_a^2)^{1/2}$ of $x$ satisfies $d\le h$, where $h=1/(1/h_1^2+\ldots+1/h_a^2)^{1/2}$, then setting $\lambda_i=d_i/h_i$ for  $i\in \{1,\ldots,a\}$, we have $\lambda_1+\ldots+\lambda_a\le 1$ by Cauchy--Schwartz. Write $y_i=(1/\lambda_i)x_i$ for each $i\in \{1,\ldots,a\}$. Then, $y_i\in h_iB^{b_i}\subset K_i$, so $x=(\lambda_1y_1,\ldots,\lambda_ay_a)\in K$ by~\eqref{eq:def ell_1 sum} as $0\in K$. Hence, $K$ contains the  Euclidean ball in $\R^{b_1}\times \ldots\times \R^{b_a}$ of radius $h$. Every facet of $K$ intersect that ball. Indeed, by assumption for every $i\in \{1,\ldots,a\}$ and every $j\in \{1,\ldots,\up_i\}$ there is $u_{ij}\in S^{b_i-1}$ such that $x_{ij}=h_iu_{ij}\in F_{ij}$. Recalling~\eqref{eq:facet formula}, for every $(j_1,\ldots,j_a)\in \{1,\ldots,\up_1\}\times \ldots\times \{1,\ldots,\up_a\}$  the vector $h^{2}(h_1^{-1}u_{1j_1},\ldots,h_a^{-1}u_{aj_a})=h^{2}(h_1^{-2}x_{1j_1},\ldots,h_a^{-2}x_{aj_a})$ belongs to $F_{j_1\ldots j_a}$ and its Euclidean length equals $h$. We thus checked that the assumptions of Lemma~\ref{lem:cones} hold for $K$, whence~\eqref{eq:surface area forluma l1 sum} follows from~\eqref{volume formula l1 sum} through  an application of of Lemma~\ref{lem:cones}.

If we assume in addition that for every $i\in \{1,\ldots,a\}$ the facets of $K_i$ are congruent to each other, then $\vol_{b_i-1}(F_{ij})=\vol_{b_i-1}(\partial K_i)/\up_i$ for every $i\in \{1,\ldots,a\}$ and $j\in \{1,\ldots,\up_i\}$. Furthermore, the facets of $K$ that are given in~\eqref{eq:facet formula} are now congruent to each other,  so $\vol_{b_1+\ldots+b_a-1}(F_{j_1\ldots j_a})=\vol_{b_1+\ldots+b_a-1}(\partial K)/(\up_1\cdots\up_a)$ for every $(j_1,\ldots,j_a)\in \{1,\ldots,\up_1\}\times \ldots\times \{1,\ldots,\up_a\}$. For every $i\in \{1,\ldots,a\}$ and $j\in \{1,\ldots,\up_i\}$ the unit outer normal to $F_{ij}$ is $u_{ij}=(u_{ij1},\ldots,u_{ijb_i})\in S^{b_i-1}$, so it follows that for every $i\in \{1,\ldots,a\}$ the area measure $\sigma_{K_i}$ is equal to $\vol_{b_i-1}(\partial K_i)/\up_i$ times the sum over $j\in \{1,\ldots,\up_i\}$ of the point mass at $u_{ij}$.  The unit outer normal to $F_{j_1\ldots j_a}$ is $h(h_1^{-1}u_{1j_1},\ldots,h_a^{-1}u_{aj_a})\in \R^{b_1}\times\ldots  \times \R^{b_a}$ for every $(j_1,\ldots,j_a)\in \{1,\ldots,\up_1\}\times \ldots\times \{1,\ldots,\up_a\}$, so it similarly follows that the area measure $\sigma_{K}$ is equal to $\vol_{b_1+\ldots+b_a-1}(\partial K)/(\up_1\cdots\up_a)$ times the sum over $(j_1,\ldots,j_a)\in \{1,\ldots,\up_1\}\times \ldots\times \{1,\ldots,\up_a\}$ of the point mass at  $h(h_1^{-1}u_{1j_1},\ldots,h_a^{-1}u_{aj_a})$.

Therefore, if we assume in addition that $\iq(K_i)=\partial_{K_i}$ for every $i\in \{1,\ldots,a\}$, then by the  criterion of~\cite{Pet61} the covariance matrix of $\sigma_{K_i}$ must be equal to $\vol_{b_i-1}(\partial K_i)/b_i$ times the identity matrix, whence: 
\begin{equation}\label{eq:bi isotropic}
\forall i\in \{1,\ldots,a\},\ \forall r,s\in \{1,\ldots, b_i\},\qquad \sum_{j=1}^{\up_i} u_{ijr}u_{ijs}=\frac{\up_i}{b_i}\d_{rs}.
\end{equation}
Next, take distinct $i,i'\in \{1,\ldots,a\}$. For every $r\in \{1,\ldots,b_i\}$ and $r'\in \{1,\ldots,b_{i'}\}$, by the above description of $\sigma_K$,  the $(r,r')$-entry of its covariance matrix equals the following quantity: 
\begin{align}\label{eq:cov K}
\frac{\vol_{b_1+\ldots+b_a-1}(\partial K)h^2}{\up_1\cdots\up_ah_ih_{i'}}\sum_{(j_1,\ldots,j_a)\in \{1,\ldots,\up_1\}\times \ldots\times \{1,\ldots,\up_a\}}  u_{ij_ir}u_{i'j_{i'}r'} =\frac{\vol_{b_1+\ldots+b_a-1}(\partial K)h^2}{\up_i\up_{i'}h_ih_{i'}}\sum_{j=1}^{\up_i}\sum_{j'=1}^{\up_{i'}} u_{ijr}u_{i'j'r'}. 
\end{align}
If we further assume that the bodies $K_1,\ldots,K_a$ are origin-symmetric, then $-u_{ij}\in \{u_{i1},\ldots,u_{i\up_i}\}$ for every  $j\in \{1,\ldots,\up_i\}$, so the right hand side of~\eqref{eq:cov K} vanishes.   At the same time, for   $i\in \{1,\ldots,a\}$ and $r,s\in \{1,\ldots,b_i\}$ the $(r,s)$-entry of  the covariance matrix of $\sigma_K$ equals the following quantity: 
\begin{multline*}
\frac{\vol_{b_1+\ldots+b_a-1}(\partial K)h^2}{\up_1\cdots\up_ah_i^2}\sum_{(j_1,\ldots,j_a)\in \{1,\ldots,\up_1\}\times \ldots\times \{1,\ldots,\up_a\}}  u_{ij_ir}u_{ij_{i}s}\\=\frac{\vol_{b_1+\ldots+b_a-1}(\partial K)h^2}{\up_ih_i^2}\sum_{j=1}^{\up_i} u_{ijr}u_{ijs}\stackrel{\eqref{eq:bi isotropic}}{=}h^2\vol_{b_1+\ldots+b_a-1}(\partial K)\frac{1}{b_ih_i^2}\d_{rs}.
\end{multline*}
Consequently, if in addition to the above assumptions  $h_i=1/\sqrt{b_i}$ for every $i\in \{1,\ldots,a\}$, then the covariance matrix of $\sigma_K$ is a multiple of the identity matrix on $\R^{b_1}\times\ldots\times \R^{b_a}$, so by~\cite{Pet61} we have $\iq(K)=\partial_K$, i.e., the first equality in~\eqref{minimum surface area formula} holds. The second equality follows by substituting~\eqref{volume formula l1 sum} and~\eqref{eq:surface area forluma l1 sum}  into the definition of $\iq(K)$ while using the assumption that  $h_i=1/\sqrt{b_i}$ for every $i\in \{1,\ldots,a\}$. 
The final asymptotic equivalence in~\eqref{minimum surface area formula} is a straightforward consequence of Stirling's formula. 
\end{proof}

To complete the proof of Theorem~\ref{thm:vertex reverse is}, it remains to show that~\eqref{eq:reverse isoperim in thm} holds for every convex polytope $K\subset \R^n$ that has $\ub$ vertices and some $\A\in \SL_n(\R)$, i.e., our goal is to  demonstrate that $\partial_K\lesssim \max\{\sqrt{n\log(\ub/n)},n\}$.

\begin{proof}[Proof of~\eqref{eq:reverse isoperim in thm}] The following volumetric estimate was proved in~\cite{BF88}  for every $n,\ub\in \N$:
\begin{equation}\label{eq:quote BF}
\forall v_1,\ldots,v_\ub\in B^n,\qquad \vol_n \big(\conv(\{v_1,\ldots,v_\ub\})\big)^{\frac{1}{n}}\lesssim \frac{1}{n}{\textstyle \sqrt{1+\log \frac{\ub}{n}}}.
\end{equation}
By John's theorem~\cite{Joh48} every origin-symmetric convex body $K\subset \R^n$ admits $\A\in \SL_n(\R)$ and $r>0$ such that  $rB^n\subset \A K\subset \sqrt{n} r B^n$.   If $K$ is furthermore a convex polytope with $\ub$ vertices for  $\ub\ge 2n$, then by~\eqref{eq:quote BF} we have  $\vol_n(\A K)^{1/n}\lesssim r\sqrt{\log (\ub/n)}/\sqrt{n}$.  At the same time, $\iq(\A K)\le n\vol_n(\A K)^{1/n}/r$ by the second part~\eqref{eq:iq for including} of Lemma~\ref{lem:cones}. When $\ub< 2^n$, the desired estimate~\eqref{eq:reverse isoperim in thm} follows by concatenating these two bounds. For $\up\ge 2^n$, the desired estimate~\eqref{eq:reverse isoperim in thm} is a special case of the reverse isoperimetric theorem~\cite{Bal91c}.
\end{proof}

\section{proof of Theorem~\ref{thm:weak reverse iso}}\label{sec:proof3}

Given $n\in \N$ and a convex body $K\subset \R^n$,  let $\lambda(K)$ be the smallest $\lambda>0$ such that there is $\f:\Omega\to \R$ that vanishes on $\partial K$ and  is smooth and  satisfies $-\Delta \f= \lambda \f$ on the interior of $K$, where $\Delta=\sum_{i=1}^n\partial^2/(\partial x_i^2)$ is the standard Laplacian on $\R^n$. The (classical and rudimentary) background  on spectral properties of the Dirichlet Laplacian  that is relevant to the ensuing reasoning can be found in~\cite{PS51,Cha84}; in particular, it has discrete pure point spectrum on the compact domain $K$, so the definition of $\lambda(K)$ makes sense, and:
\begin{equation}\label{eq:hpomogeneity of lambda}
\forall s\in \R\setminus \{0\},\qquad \lambda(sK)=\frac{\lambda(K)}{s^2}.
\end{equation}

The main result of this section is:

\begin{theorem}\label{thm:spectral} Fix $n,\up\in \N$ with $\up\ge n+1$. Suppose that $K\subset \R^n$ is a convex polytope   that has $\up$ facets. Then, there exists a positive definite invertible matrix $\B\in \GL_n(\R)$ such that:
\begin{equation}\label{eq:spectal two parts}
\vol_n(\B K)^{\frac{1}{n}}\lesssim \sqrt{\frac{\up}{n}}\qquad\mathrm{and}\qquad \lambda(\B K)\lesssim \up. 
\end{equation}
\end{theorem}  

Prior to proving Theorem~\ref{thm:spectral}, we will next assume its validity and explain how it implies Theorem~\ref{thm:weak reverse iso}: 
\begin{proof}[Deduction of Theorem~\ref{thm:weak reverse iso} from Theorem~\ref{thm:spectral}] Note first that the  following quick consequence of Fubini  holds for every compact subset $E$ of $\R^n$ that satisfies  $\vol_n(E)>0$:
\begin{equation}\label{eq:fubini E+E}
\sup_{x\in \R^n} \vol_n\big (E\cap (x-E)\big)\ge \frac{\vol_n(E)^2}{\vol_n(E+E)}.
\end{equation}
Indeed, $E\cap (x-E)\neq \emptyset$ if and only if $x\in E+E$, whence: 
$$
\vol_n(E+E) \sup_{x\in \R^n} \vol_n\big (E\cap (x-E)\big)\ge \int_{\R^n} \vol_n\big (E\cap (x-E)\big)\ud x=\iint_{\R^n\times \R^n} \1_E(y)\1_{E}(x-y)\ud y\ud x=\vol_n(E)^2.
$$

Let $K\subset \R^n$ be a convex body. Then,  $\sqrt[n]{\vol_n(K\cap (x-K))}\ge \sqrt[n]{\vol_n(K)}/2$ for some  $x\in \R^n$  by an application of~\eqref{eq:fubini E+E} with $E=K$ and the fact that the convexity of $K$ can be restated as  $K+K= 2K$. Denote: 
\begin{equation}\label{eq:def K'K''}
K'\eqdef -\frac12 x+K\qquad\mathrm{and}\qquad  K''\eqdef K'\cap(-K').
\end{equation}
Then $K''$ is an origin-symmetric convex body that satisfies:
\begin{equation}\label{eq:vol K''}
\vol_n(K'')^{\frac{1}{n}}\gtrsim \vol_n(K)^{\frac{1}{n}}. 
\end{equation}
 If furthermore $K$ is a convex polytope that has $\up$ facets (per the setting of Theorem~\ref{thm:weak reverse iso} ), then $K''$ is an origin-symmetric convex polytope that has at most $2\up$ facets.

 Apply Theorem~\ref{thm:spectral} to $K''$, thus obtaining  a positive definite matrix $\B\in \GL_n(\R)$ that satisfies:
\begin{equation}\label{eq:spectal two parts''}
\vol_n(\B K'')^{\frac{1}{n}}\lesssim \sqrt{\frac{\up}{n}}\qquad\mathrm{and}\qquad \lambda(\B K'')\lesssim \up. 
\end{equation}
As $\B$  is positive definite and invertible, we may consider the matrix $\A= (\det\B)^{-1/n}\B\in \SL_n(\R)$. Now:  
\begin{equation}\label{eq:spectal product}
\lambda(\A K'')\vol_n(\A K'')^{\frac{2}{n}}\stackrel{\eqref{eq:hpomogeneity of lambda}}{=}\lambda(\B K'')\vol_n(\B K'')^{\frac{2}{n}}\stackrel{\eqref{eq:spectal two parts''}}{\lesssim}\frac{\up^2}{n}.  
\end{equation}

Following~\cite{Nao24}, let $\Ch \A K''\subset \A K''$ denote the Cheeger body of $\A K''$, namely, it is the unique measurable subset of $\A K''$ that satisfies $\Per(\Ch \A K'')/\vol_n(\Ch \A K'')\le \Per(E)/\vol_n(E)$ for every measurable $E\subset \A K''$ with $\vol_n(E)>0$, where $\Per(\cdot)$ denotes perimeter in the sense of Caccioppoli and de~Giorgi (a thorough treatment of this notion of perimeter can be found in e.g.~\cite{AFP00}). By~\cite{AC09}, such a minimizer exists and it is indeed unique, and furthermore it is a convex subset of $\A K''$. The aforementioned  (substantial) theorem of~\cite{AC09} that this minimizer is unique implies in particular that since $\A K''$ is origin-symmetric, its Cheeger body  $\Ch \A K''$ is also origin-symmetric. By substituting~\eqref{eq:spectal product} into  equation~(1.62)  of~\cite{Nao24}, we see that:
\begin{equation}\label{eq:quote cheeger inequality}
\frac{\iq(\Ch \A K'')}{\sqrt{n}}\left(\frac{\vol_n(\A K'')}{\vol_n(\Ch \A K'')}\right)^{\frac{1}{n}}\lesssim \frac{\up}{n}.
\end{equation}
By the isoperimetric theorem~\eqref{eq:quote isoperimetric theorem} applied to $\Ch \A K''$  and the inclusion  $\Ch \A K''\subset \A K''$, we have:
\begin{align}\label{eq:max form}
\begin{split}
\frac{\iq(\Ch \A K'')}{\sqrt{n}}&\left(\frac{\vol_n(\A K'')}{\vol_n(\Ch \A K'')}\right)^{\frac{1}{n}}\gtrsim \max \left\{\frac{\iq(\Ch \A K'')}{\sqrt{n}},\left(\frac{\vol_n(\A K'')}{\vol_n(\Ch \A K'')}\right)^{\frac{1}{n}}\right\}\\&=\max \left\{\frac{\iq(\Ch \A K'')}{\sqrt{n}},\left(\frac{\vol_n(K'')}{\vol_n(\Ch \A K'')}\right)^{\frac{1}{n}}\right\} \stackrel{\eqref{eq:vol K''}}{\gtrsim} \max \left\{\frac{\iq(\Ch \A K'')}{\sqrt{n}},\left(\frac{\vol_n(K)}{\vol_n(\Ch \A K'')}\right)^{\frac{1}{n}}\right\}, 
\end{split}
\end{align}
where the second step of~\eqref{eq:max form} holds because  $\A$ is volume-preserving.  

Finally, if we choose $z= -(1/2)\A x$ and $L= \Ch \A K''$, then $L$ is an origin-symmetric convex body, thanks to the definitions~\eqref{eq:def K'K''} we know that $L$ is a subset of $z+(\A K)\cap (\A x-\A K)\subset z+\A K$, and by  combining~\eqref{eq:quote cheeger inequality} with~\eqref{eq:max form} we conclude that the desired requirements~\eqref{eq:desired conditions for isomorphic in thm} of Theorem~\ref{thm:weak reverse iso} indeed hold. 
\end{proof}

\begin{remark}{\em  The above proof of Theorem~\ref{thm:weak reverse iso} demonstrates that for every convex polytope $K\subset \R^n$ that has $\up$ facets there exists $\A\in \SL_n(\R)$ such that  $\lambda(\A K)\vol_n(K)^{2/n}\lesssim \up^2/n$.  The previously best-known bound here~\cite[page~51]{Nao24} (for any convex body $K$) is $\lambda(\A K)\vol_n(K)^{2/n}\lesssim n(\log n)^2$. Thus, we obtain an asymptotic improvement whenever $\up=o(n\log n)$ and not only when $\up =O(n)$.  
}
\end{remark}

\begin{proof}[Proof of Theorem~\ref{thm:spectral}] Because $K$ is origin-symmetric and bounded, its number of facets $\up\ge 2n$ is even, so we may write $\up=2m$ for some integer $m\ge n$. Thus, we can fix $y_1,\ldots,y_m\in \R^n\setminus\{0\}$ such that  
\begin{equation}\label{eq:new K slabs}
K=\big\{x\in \R^n:\ \max_{i\in \m} |\langle x,y_i\rangle|\le 1\big\}.
\end{equation} 

As $K$ has nonempty interior, $y_1,\ldots,y_m$ span $\R^n$, so if we let $\sY\in \M_{n\times m}(\R)$ be the $n$-by-$m$ matrix whose columns are $y_1,\ldots, y_m$, then the rank of $\sY$ equals $n$. Consequently, $n$-by-$n$ matrix $\sY\sY^*\in \M_n(\R)$ is positive semidefinite and invertible, so we can define  $\sfC=(\sY\sY^*)^{-1/2}\sY\in \M_{n\times m}(\R)$. It is straightforward to verify that $\mathsf{C}\mathsf{C^*}$ is the $n$-by-$n$ identity matrix $\Id_n$. In other words, the rows $\ur_1,\ldots,\ur_n\in \R^m$ of $\mathsf{C}$ are orthonormal, so there are $\ur_{n+1},\ldots,\ur_{m}\in \R^m$ such that $\ur_1,\ldots,\ur_m$ is an orthonormal basis of $\R^m$.  Let $\V\in \O_m$ denote the orthogonal matrix whose rows are $\ur_1,\ldots,\ur_m$. Let   $\ug_1,\ldots,\ug_m\in \R^m$ be the columns of $\V$, which  are also an orthonormal basis of $\R^m$. Letting $\mathsf{R}_{m\to n}:\R^m\to \R^n$ denote the restriction operator from $\R^m$ to $\R^n$, i.e., $\mathsf{R}_{m\to n}w=(w_1,\ldots,w_n)$ for every $w=(w_1,\ldots,w_m)\in \R^m$, observe that:
\begin{equation}\label{eq:CK}
\B K=\big\{x\in \R^n:\ \max_{i\in \m} |\langle x,\mathsf{R}_{m\to n}\ug_i\rangle| \le 1\big\}\qquad\mathrm{where}\qquad \B\eqdef \left(\sY\sY^*\right)^{\frac12}.
\end{equation}
Indeed, $\B K=\{z\in \R^n:\ \max_{i\in \m} |\langle z,\B^{-1}y_i\rangle| \le 1\}=\{z\in \R^n:\ \max_{i\in \m} |\langle z,\mathsf{R}_{m\to n}\B^{-1}y_i\rangle| \le 1\}$ by~\eqref{eq:new K slabs}, and $\B^{-1}y_i\in \R^m$ is the $i$'th column of $\mathsf{C}=\B^{-1}\sY$ for every $i\in \m$, which equals by definition $\ug_i$. 

As $\B K$ has $2m$ facets, it follows from~\eqref{eq:CK}  that $\mathsf{R}_{m\to n}\ug_i\neq 0$. We can therefore denote
\begin{equation}\label{eq:def ciui}
\forall i\in \m,\qquad c_i\eqdef \|\mathsf{R}_{m\to n}\ug_i\|_{\!\ell_{\!\!2}^n}\qquad\mathrm{and}\qquad u_i\eqdef \frac{1}{\sqrt{c_i}}\mathsf{R}_{m\to n}\ug_i\in S^{n-1}. 
\end{equation}
Using this notation, we can rewrite~\eqref{eq:CK} as follows:
\begin{equation}\label{eq:write AK using ui}
\B K=\big\{x\in \R^n:\ \max_{i\in \m} c_i\langle x,u_i\rangle^2\le 1\big\}=\bigcap_{i=1}^m \bigg\{x\in \R^n:\ -\frac{1}{\sqrt{c_i}}\le \langle x,u_i\rangle\le \frac{1}{\sqrt{c_i}}\bigg\}. 
\end{equation}
Furthermore, because $\ug_1,\ldots,\ug_m$ is an orthonormal basis of $\R^m$, we have 
$$
\forall x\in \R^n, \qquad \sum_{i=1}^n c_i \langle x,u_i\rangle^2\stackrel{\eqref{eq:def ciui}}{=} \sum_{i=1}^n  \langle \mathsf{R}_{m\to n}^* x,\ug_i\rangle^2=\|\mathsf{R}_{m\to n}^* x\|_{\!\ell_{\!\!2}^m}^2=\|x\|_{\!\ell_{\!\!2}^n}^2. 
$$
In other words, the following matrix identity holds: 
\begin{equation}\label{eq:decomposition of the identity}
\sum_{i=1}^m c_i u_i\otimes u_i=\I_n. 
\end{equation}

Thanks to~\eqref{eq:decomposition of the identity}, according the geometric form of the
 Brascamp--Lieb inequality~\cite{BL76} that was first formulated in~\cite{Bal89}
(see also e.g.~Theorem~2 in~\cite{Bal01}), if $f_1,\ldots,f_m:\R\to [0,\infty)$ are measurable, then:
\begin{equation}\label{eq:quote BL}
 \int_{\R^n}\bigg(\prod_{i=1}^m f_i(\langle x,u_i\rangle)^{c_i}\bigg)\ud x\leq
 \prod_{i=1}^m\bigg(\int_{-\infty}^\infty f_i(t)\ud t\bigg)^{c_i}.
\end{equation}
Consequently,
\begin{equation}\label{eq:use BL}
 \vol_n(\B K) \stackrel{\eqref{eq:write AK using ui}}{=}\int_{\R^n}\bigg(\prod_{i=1}^m \1_{\big[-\frac{1}{\sqrt{c_i}},\frac{1}{\sqrt{c_i}}\big]}(\langle x,u_i\rangle)^{c_i}\bigg)\ud x \stackrel{\eqref{eq:quote BL}}{\le}
 \prod_{i=1}^m\left(\frac2{\sqrt{c_i}}\right)^{c_i}.
\end{equation}
By taking the trace of~\eqref{eq:decomposition of the identity} we get $\sum_{i=1}^m c_i=n$. The maximum of the right hand side of~\eqref{eq:use BL} over all possible $c_1,\ldots,c_m>0$ satisfying $\sum_{i=1}^m c_i=n$ is attained when $c_1=\ldots=c_m=n/m$. Hence~\eqref{eq:use BL}  implies that:
$$
\vol_n(\B K)^{\frac{1}{n}}\le 2\sqrt{\frac{m}{n}}=\sqrt{\frac{2\up}{n}},
$$
thus proving the first part of~\eqref{eq:spectal two parts}.

To prove the second part of~\eqref{eq:spectal two parts}, thus concluding the proof of Theorem~\ref{thm:spectral}, define $\f:K\to \R$ by setting:
\begin{equation}\label{eq:def our test function}
\forall x\in \R^n,\qquad \f(x)\eqdef =\prod_{i=1}^m\left(1-c_i\langle x,u_i\rangle^2\right)\stackrel{\eqref{eq:def ciui}}{=}\prod_{i=1}^m\left(1-\langle x,\mathsf{R}_{m\to n}\ug_i\rangle^2\right).
\end{equation}
Then, $\f$ is smooth (it is a polynomial), and thanks to~\eqref{eq:write AK using ui} it vanishes on the boundary of $K$. The standard (see e.g.~\cite{Cha84}) Rayleigh quotient
characterization of the smallest nonzero eigenvalue of the negative Laplacian $-\Delta$ on $K$ with Dirichlet boundary conditions therefore gives:
\begin{equation}\label{eq:Rayleigh quotient}
\lambda(\B K)\le \frac{\int_{\B K}\|\nabla\f(x)\|_{\!\ell_{\!\!2}^n}^2\ud x}{\int_{\B K} \f(x)^2\ud x}.
\end{equation}

To treat the numerator in~\eqref{eq:Rayleigh quotient}, define $\Phi:\R^m\to \R$ by setting:
\begin{equation}\label{eq:def Phi}
\forall z=(z_1,\ldots,z_m)\in \R^n,\qquad \Phi(x)\eqdef \prod_{i=1}^m \big(1-z_i^2\big). 
\end{equation}
Recalling that $\ug_1,\ldots,\ug_m$ are the columns of the  orthogonal matrix $\V\in \O_m$, i.e., $\ug_i=\V e_i$ for $i\in \m$, where $e_1,\ldots,e_m$ is the standard coordinate basis of $\R^m$, for every $x\in \R^n$ we  have:
$$
\f(x)=\Phi\big(\langle x, \mathsf{R}_{m\to n}\V e_1\rangle,\ldots,\langle x, \mathsf{R}_{m\to n}\V e_m\rangle\big)=\Phi\big(\langle \V^*\mathsf{R}_{m\to n}^*x, e_1\rangle,\ldots,\langle \V^*\mathsf{R}_{m\to n}^*x, e_m\rangle\big)=\Phi(\V^*\mathsf{R}_{m\to n}^*x). 
$$
Hence, $\nabla\f (x)= \mathsf{R}_{m\to n}\V \nabla\Phi(\mathsf{R}_{m\to n}\V x)= \mathsf{R}_{m\to n}\V \nabla\Phi(\V^*\mathsf{R}_{m\to n}^*x)$ by the chain rule. Consequently,  
\begin{align}\label{eq:nabla bound}
\begin{split}
&\!\!\!\|\nabla\f(x)\|_{\!\ell_{\!\!2}^n}^2=\|\mathsf{R}_{m\to n}\V \nabla\Phi(\V^*\mathsf{R}_{m\to n}^*x)\|_{\!\ell_{\!\!2}^n}^2\le \|\V \nabla\Phi(\V^*\mathsf{R}_{m\to n}^*x)\|_{\!\ell_{\!\!2}^m}^2=\|\nabla\Phi(\V^*\mathsf{R}_{m\to n}^*x)\|_{\!\ell_{\!\!2}^m}^2\\& \!\!\!\!\!\! \stackrel{\eqref{eq:def Phi}}{=}
\sum_{i=1}^m 4\langle \V^*\mathsf{R}_{m\to n}^*x,e_i\rangle^2\!\!\!\!\! \prod_{j\in \m\setminus \{i\}} \big(1-\langle \V^*\mathsf{R}_{m\to n}^*x,e_j\rangle^2\big)^2= 4 \sum_{i=1}^m  c_i \langle x,u_i\rangle^2 \!\!\!\!\! \prod_{j\in \m\setminus \{i\}} \big(1-c_j \langle x,u_j\rangle^2\big)^2.
\end{split}
\end{align}
Consequently, if for each $i\in \m$ we define
\begin{equation}\label{eq:def gi}
\forall t\in \Big[-\frac{1}{\sqrt{c_i}},\frac{1}{\sqrt{c_i}}\Big],\qquad g_i(t)\eqdef \int_{(tu_i+u_i^\bot)\cap \B K}
\bigg(\prod_{j\in \m\setminus \{i\}}\big(1-c_j\langle x,u_j\rangle^2\big)^2\bigg)dx, 
\end{equation}
then by Fubini we have 
\begin{equation}\label{eq:use fubini 1}
\int_{\B K}\|\nabla\f(x)\|_{\!\ell_{\!\!2}^n}^2\ud x\stackrel{\eqref{eq:nabla bound}\wedge\eqref{eq:def gi}\wedge \eqref{eq:write AK using ui}}{\le} \sum_{i=1}^m 4c_i \int_{-\frac{1}{\sqrt{c_i}}}^{\frac{1}{\sqrt{c_i}}}t^2g_i(t)\ud t.
\end{equation}

To estimate the right hand side of~\eqref{eq:use fubini 1}, observe that for every $i\in \m$ the function 
$$
\forall x\in \R^n,\qquad h_i(x)\eqdef  \left\{ \begin{array}{cl} \sum_{j\in \m\setminus \{i\}} \log \big(1-c_j\langle x,u_j\rangle^2\big)&\mathrm{if\ } x\in \B K,\\ 0&\mathrm{if\ } x\in \R^n\setminus\B K,\end{array}\right.
$$
is well-defined by~\eqref{eq:write AK using ui}, and concave on $\R^n$ since $\B K$ is convex and  $(t\in [-1,1])\mapsto \log(1-t^2)$ is concave (its second derivative equals $-2(1+t^2)/(1-t^2)^2$). The function $g_i$ that is defined in~\eqref{eq:def gi} is thus the marginal of the log-concave function $e^{h_i}$ in the direction of $u_i$, whence $g_i$ is log-concave on $[-1/\sqrt{c_i},1/\sqrt{c_i}]$ by~\cite{Pre73}. Also, $g_i$ is  nonnegative and even, because $\B K$ is origin-symmetric, so it is nonincreasing on $[0,1/\sqrt{c_i}]$ (if $0\le s\le t$, then write $s=\lambda t+(1-\lambda)(-t)$ for $\lambda=(s+t)/(2t)\in [0,1]$ and estimate $g(s)\ge g(t)^\lambda g(-t)^{1-\lambda}=g(t)$ by log-concavity). For fixed $i\in \m$, we can therefore bound the $i$'th summand in~\eqref{eq:use fubini 1} as follows: 
\begin{align}\label{eq:first integral split}
\begin{split}
4c_i \int_{-\frac{1}{\sqrt{c_i}}}^{\frac{1}{\sqrt{c_i}}}t^2g_i(t)\ud t&= 8c_i\int_0^{\frac{1}{2\sqrt{c_i}}}t^2g_i(t)\ud t+ 8c_i\int^{\frac{1}{\sqrt{c_i}}}_{\frac{1}{2\sqrt{c_i}}}t^2g_i(t)\ud t\\ &\le \bigg(\max_{0\le s\le \frac{1}{2\sqrt{c_i}}}\frac{8c_i s^2}{(1-c_i s^2)^2}\bigg) \int_0^{\frac{1}{2\sqrt{c_i}}}(1-c_it^2)^2g_i(t)\ud t+ 8c_i g_i\Big(\frac{1}{2\sqrt{c_i}}\Big)\int^{\frac{1}{\sqrt{c_i}}}_{\frac{1}{2\sqrt{c_i}}}t^2\ud t\\
&= \frac{32}{9}\int_0^{\frac{1}{2\sqrt{c_i}}}(1-c_it^2)^2g_i(t)\ud t+\frac{7g_i\Big(\frac{1}{2\sqrt{c_i}}\Big)}{3\sqrt{c_i}},
\end{split}
\end{align}
where the final step of~\eqref{eq:first integral split} holds because the maximum that appears in~\eqref{eq:first integral split} is attained at $s=1/(2\sqrt{c_i})$.
The final term in~\eqref{eq:first integral split}  can be bounded from above as follows: 
\begin{equation}\label{eq:bound g at cutoff}
\int_0^{\frac{1}{2\sqrt{c_i}}}(1-c_it^2)^2g_i(t)\ud t\ge g_i\Big(\frac{1}{2\sqrt{c_i}}\Big)\int_0^{\frac{1}{2\sqrt{c_i}}}(1-c_it^2)^2\ud t=\frac{203g_i\Big(\frac{1}{2\sqrt{c_i}}\Big)}{480\sqrt{c_i}}. 
\end{equation}
By combining~\eqref{eq:first integral split} and~\eqref{eq:bound g at cutoff} we conclude that the following inequality  holds for every $i\in \m$:  
\begin{align}\label{eq:fubini2}
\begin{split}
4c_i \int_{-\frac{1}{\sqrt{c_i}}}^{\frac{1}{\sqrt{c_i}}}t^2g_i(t)\ud t&< 10\int_0^{\frac{1}{2\sqrt{c_i}}}(1-c_it^2)^2g_i(t)\ud t \le 10\int_{0}^{\frac{1}{\sqrt{c_i}}} (1-c_it^2)^2g_i(t)\ud t\\& =  5\int_{-\frac{1}{\sqrt{c_i}}}^{\frac{1}{\sqrt{c_i}}} (1-c_it^2)^2g_i(t)\ud t
\stackrel{\eqref{eq:write AK using ui}\wedge \eqref{eq:def our test function}\wedge \eqref{eq:def gi}}{=}5\int_{\B K} \f(x)^2\ud x,
\end{split}
\end{align}
where the last step of~\eqref{eq:fubini2} is another application of Fubini. By substituting~\eqref{eq:fubini2} into~\eqref{eq:use fubini 1} and then substituting the resulting estimate into~\eqref{eq:Rayleigh quotient}, we get that $\lambda(\B K)\le 5m\asymp \up$, i.e., the second part of~\eqref{eq:spectal two parts} holds.    
\end{proof}

\bibliographystyle{abbrv}
\bibliography{BBN}

 \end{document}